%% file: genericsystems.tex
\documentclass[10.5pt]{article}
\usepackage{amsfonts}
\usepackage{graphicx}
\usepackage{epstopdf}
\usepackage{pgfplots}
\usepackage{tikz}
\usepackage{algorithm}
\usepackage{algpseudocode}
\usepackage{amsmath}
\usepackage{amsthm}
\usepackage[margin=1.3in]{geometry}
\newtheorem{example}{Example}
\newtheorem{theorem}{Theorem}

\newtheorem{definition}{Definition}

\newcommand{\nulll}{\textup{null}}

\newcommand{\rank}{\textup{rank}}

\newcommand{\Res}{\textup{Res}}

\newcommand{\I}{I}
\usepackage{url}
\newcommand{\OO}{\mathcal{O}}
\newcommand{\Ints}{\mathbb{Z}_{\geq 0}}
\newcommand{\Z}{\mathbb{Z}}

\newcommand{\C}{\mathbb{C}}
\newcommand{\PP}{\mathbb{P}}

\newcommand{\V}{\mathbb{V}}
\newcommand{\x}{x}
\newcommand{\z}{z}

\newcommand{\R}{\mathbb{R}}
\usepackage{tikz}
\usepackage{pgfplots}

\newcommand{\Span}{\textup{span}}

\usepackage{kbordermatrix}
\usepackage{tikz-cd}
\pgfplotsset{
  log x ticks with fixed point/.style={
      xticklabel={
        \pgfkeys{/pgf/fpu=true}
        \pgfmathparse{exp(\tick)}%
        \pgfmathprintnumber[fixed relative, precision=3]{\pgfmathresult}
        \pgfkeys{/pgf/fpu=false}
      }
  }}
\begin{document}
\title{A Stabilized Normal Form Algorithm for Generic Systems of Polynomial Equations}
\author{Simon Telen, Marc Van Barel\thanks{Supported by
  the Research Council KU Leuven,
PF/10/002 (Optimization in Engineering Center (OPTEC)),
C1-project (Numerical Linear Algebra and Polynomial Computations),
by
the Fund for Scientific Research--Flanders (Belgium),
G.0828.14N (Multivariate polynomial and rational interpolation and approximation),
   and by
   the Interuniversity Attraction Poles Programme, initiated by the Belgian State,  Science Policy Office,
   Belgian Network DYSCO (Dynamical Systems, Control, and Optimization).
}
}
\maketitle

\begin{abstract}
We propose a numerical linear algebra based method to find the multiplication operators of the quotient ring $\C[\x]/\I$ associated to a zero-dimensional ideal $\I$ generated by $n$ $\C$-polynomials in $n$ variables. We assume that the polynomials are generic in the sense that the number of solutions in $\C^n$ equals the B\'ezout number. The main contribution of this paper is an automated choice of basis for $\C[\x]/\I$, which is crucial for the feasibility of normal form methods in finite precision arithmetic. This choice is based on numerical linear algebra techniques and it depends on the given generators of $\I$.  
\end{abstract}

\section{Introduction}
Consider the following problem. Given $n$ polynomials $f_1, \ldots, f_n \in k[x_1,\ldots,x_n]$ with $k$ an algebraically closed field, find all the points $\x \in k^n$ where they all vanish: $f_1(\x) = \ldots = f_n(\x) = 0$. Here, we will work over the complex numbers $k = \C$. The ring of all polynomials in the $n$ variables $x_1, \ldots, x_n$ with coefficients in $\C$ is denoted by $\C[x_1,\ldots,x_n]$. For short, we will denote $\x = (x_1,\ldots,x_n)$ and an element $f \in \C[\x]$ can be written as 
$$ f = \sum_{\alpha \in \Ints^n} c_{\alpha} \x^{\alpha}$$
where we used the short notation $\x^{\alpha} = x_1^{\alpha_1} \cdots x_n^{\alpha_n}$. The \textit{support} $S(f)$ of $f$ is defined as
$$S(f) = \{ \alpha \in \Ints^n: c_{\alpha} \neq 0 \}.$$
A set of $n$ polynomials $\{f_1, \ldots, f_n \} \subset \C[\x]$ defines a \textit{square ideal} 
$$\I = \langle f_1, \ldots, f_n \rangle = \{g_1f_1+ \ldots + g_nf_n: g_1,\ldots,g_n \in \C[\x] \} \subset \C[\x].$$ 
The \textit{affine variety} associated to $I$ is 
$$ \V(\I) = \{ \x \in \C^n: f(\x) = 0, \forall f \in \I\} =  \{ \x \in \C^n: f_1(\x) = \ldots = f_n(\x) = 0 \}.$$
In this paper, we assume that the variety $\V(\I)$ consists of finitely many points $\{\z_1, \ldots, \z_N \} \subset \C^ n$. Such a variety is called \textit{0-dimensional}. \\

A well known result in algebraic geometry states that the quotient ring $k[x_1, \ldots, x_n]/\I$ with $\I \subset k[x_1, \ldots, x_n]$ a 0-dimensional ideal and $k$ an algebraically closed field is isomorphic as a $k$-algebra to a finite dimensional $k$-vectorspace $V$ with multiplication defined by a pairwise commuting set of $n$ square matrices over $k$. This set of matrices corresponds to a set of generators of $k[x_1, \ldots, x_n]/\I$ and the size of each matrix is equal to the number of points in $\V(\I) \subset k^n$, counting multiplicities. Once the (generating) multiplication matrices are known in some basis, we can answer several questions about the variety $\V(\I)$. For example, we can retrieve the solutions of the system by computing their eigenstructure and we can evaluate any polynomial on $\V(\I)$. Our goal is to compute the multiplication matrices in a numerically stable way for square ideals satisfying some genericity assumptions. \\

There are many approaches to the problem of solving systems of polynomial equations. The different methods are often subdivided in homotopy methods, subdivision methods and algebraic methods. Homotopy continuation uses Newton iteration to track solution paths, starting from a simple initial system and gradually transforming it into the target system. These ideas have led to highly successful solvers \cite{bates2013numerically,verschelde1999algorithm}. However, performing some numerical experiments one observes that for large systems some solutions might be lost along the way. The continuation gives up on certain paths when, for example, they seem to be diverging to infinity or they enter an ill-conditioned region. Normal form algorithms belong to the category of algebraic methods. The earliest versions of these algorithms use Groebner bases \cite{cox1,cox2} and doing so they make an implicit choice of basis for $\C[\x]/\I$. It turns out that these methods are numerically unstable and infeasible for large systems of equations (high degree, many variables). More recent algorithms are based on \textit{border bases} \cite{mourrain1999new,stetter,mourrain2007pythagore}. Essentially, they fix a basis $\OO$ for $\C[\x]/\I$ and construct the multiplication matrices of the coordinate functions by calculating the normal forms of $x_1 \cdot \OO$, \ldots, $x_n \cdot \OO$ with respect to $\OO$. Border bases are a generalization of Groebner bases and they can be used to enhance the numerical stability of normal form algorithms. However, there are no algorithms that make a choice of $\OO$ based on the conditioning of the normal form computation problem. This is mentioned as an open problem in \cite{mourrain2007pythagore}. In this paper we present such an algorithm for generic systems that makes an automatic choice of $\OO$, which does not necessarily correspond to a Groebner basis, nor to a border basis. What is meant by `generic systems' is explained in Section \ref{sec:genericsystems}. The goal is to cover the generic, dense case to illustrate the effectiveness of the idea. The connection with resultant algorithms for dense systems is established. This suggests that the techniques can be generalized to  sparse systems of equations. Such a generalization will follow from the sparse variant of the Macaulay resultant algorithm, see for instance \cite{emir1}.  \\

In the following section we discuss our genericity assumptions and some properties of the systems that satisfy them. Section \ref{sec:multmtces} briefly reviews the multiplication maps in $\C[\x]/\I$ and their properties. We give a short motivation in Section \ref{sec:motivation} by discussing some aspects of Macaulay's resultant construction and border bases algorithms that are generalized in our approach. In Section \ref{sec:macmtx} we introduce a construction that we call \textit{Macaulay matrices}. Section \ref{sec:normalform} presents the algorithm and some connections with border bases and Macaulay resultants. In the final section we present some numerical experiments.

\section{Generic total degree systems} \label{sec:genericsystems}
We say that a polynomial $f \in \C[\x] \backslash \{0\}$ is of degree $d$ if 
$$ \max_{\alpha \in S(f)} | \alpha| = d,$$ where $|\alpha| = \alpha_1+ \ldots + \alpha_n$. We denote $\deg(f) = d$. Accordingly, we say that a square polynomial system in $n$ variables given by $\{f_1,\ldots,f_n \}$ is of degree $(d_1, \ldots, d_n)$ if $\deg(f_i) = d_i, i=1, \ldots,n$. A polynomial $f \in \C[\x]\backslash \{0\}$ is called \textit{homogeneous} of degree $d$ if $|\alpha | = d, \forall \alpha \in S(f)$. \\
Consider the \textit{projective $n$-space} $$\PP^n = (\C^{n+1} \backslash \{0\} )/\sim,$$
where $(a_0, \ldots, a_n) \sim (b_0, \ldots, b_n)$ iff $a_i = \lambda b_i, i=0, \ldots, n, \lambda \in \C \backslash \{0\}$. We can interpret $\PP^n$ as the union of $n+1$ copies of $\C^n$, each of them given by putting one of the coordinates equal to 1. We will also think of $\PP^n$ as the union of $\C^n$ corresponding to $x_0 =1$ and the set $\{x_0=0 \}$, called the \textit{hyperplane at infinity}. For more on projective space, see \cite{cox1}. Note that the equation $f = 0$ with $f \in \C[x_0, \ldots, x_n]$ is well defined over $\PP^n$ if and only if $f$ is homogeneous. Starting from a polynomial $f \in \C[\x]$ in $n$ variables of degree $d$, we can obtain a homogeneous polynomial $f^h \in \C[x_0,\ldots, x_n]$, called the \textit{homogenization} of $f$ as 
$$ f^h = x_0^d f \left ( \frac{x_1}{x_0}, \ldots, \frac{x_n}{x_0} \right ).$$
The following theorem was proved by \'Etienne B\'ezout for the intersection of algebraic plane curves in $\PP^2$. The generalization is often referred to as B\'ezout's theorem.
\begin{theorem}[B\'ezout] \label{bezout}
A system of $n$ homogeneous equations of degree $(d_1,\ldots,d_n)$ in $n+1$ variables with a finite number of solutions in $\PP^n$ has exactly $d_1\cdots d_n$ solutions in $\PP^n$, counting multiplicities.
	\end{theorem}
	\begin{proof}
		The theorem is a corollary of Theorem 7.7 in \cite{hartshorne}. 
	\end{proof}
It is not difficult to show that for \textit{almost all} systems with degree $(d_1, \ldots,d_n)$, all $d_1 \cdots d_n$ solutions lie in the overlapping part of the affine charts of $\PP^n$ \cite{cox2}. Hence, if the $n$ homogeneous equations in $n+1$ variables of Theorem \ref{bezout} are the homogenizations of $n$ affine equations $f_1 = \ldots = f_n = 0$ in $n$ variables, all of the $d_1 \cdots d_n$ solutions correspond to points in $\C^n \subset \PP^n$. \\

The kind of systems that we consider in this paper are the ones that satisfy the assumption of B\'ezout's theorem. Namely, we assume that the homogenized equations $f_1^h = \ldots = f_n^h = 0$ have a finite number of solutions in $\PP^n$. We denote $\overline{I} = \langle f_1^h,\ldots, f_n^h \rangle$ and $\V(\overline{\I}) = \{\x \in \PP^n: f_1^h(\x) = \ldots = f_n^h(\x) = 0 \}$. Furthermore, we assume that none of the solutions lie on the hyperplane at infinity. Note that this last assumption is not really restrictive: a random linear change of projective coordinates will move all of the solutions away from the hyperplane $\{x_0=0\}$ with probability 1.

\section{Multiplication in $\C[\x]/\I$} \label{sec:multmtces}
In this section we briefly review the $\C$-algebra structure of the quotient ring $\C[\x]/\I$ and the properties of multiplication in this ring. For an extensive treatment one can consult \cite{cox1, cox2, stetter, elkadi_introduction_2007}. Consider the following equivalence relation on $\C[\x]$:
$$ f \sim g \Leftrightarrow f-g \in \I. $$
Now, every polynomial $f \in \C[\x]$ defines a residue class $[f] = f + \I$ with respect to $\sim$. We call the polynomial $f$ a \textit{representative} of the residue class $[f]$. The set of all such residue classes is the \textit{quotient ring} $\C[\x]/ \I$. Note that $[0]= \I$. One can check that the scalar multiplication and addition operations
\begin{equation} \label{vecops}
\alpha [f] = [\alpha f], \qquad [f]+[g] = [f+g]
\end{equation}
with $\alpha \in \C$ and $f,g \in \C[\x]$ are well defined. This implies that $\C[\x]/ \I$ is a vector space. Moreover, to show that $\C[\x]/ \I$ is a $\C$-algebra, it can be checked that multiplication $$[f] \cdot [g] = [fg]$$ is well defined. The following theorem allows us to describe these operations on $\C[\x]/\I$ using linear algebra.
\begin{theorem} \label{theo:dim}
For a zero-dimensional ideal $\I$, the dimension of $\C[\x]/ \I$ as a vector space is equal to the number of points in $\V(\I) \subset \C^n$, counting multiplicities. 
\end{theorem} 
\begin{proof}
For the proof of this theorem we refer to \cite{cox2}.
\end{proof}
We now consider the map $m_f: \C[\x]/\I \rightarrow \C[\x]/\I$ given by $$m_f([g]) = [f] \cdot [g] = [fg], \forall g \in \C[\x].$$ This map is linear, so once we choose a basis $\OO$ for $\C[\x]/\I$, it can be represented by an $N \times N$ matrix, where $N$ is the number of solutions (counting multiplicities). Under our genericity assumptions, $N$ is the B\'ezout number: $N = \prod_{i=1}^n d_i$. Once we have fixed a basis of $\C[\x]/\I$, we will no longer make a distinction between the map $m_f$ and its matrix representation. The matrix representing multiplication by $f$ is called a \textit{multiplication matrix} of $f$. Its eigenstructure has the following remarkable properties.
\begin{theorem} \label{theo:evals}
Let $\I$ be a zero-dimensional ideal in $\C[\x]$ and let $m_f$ be the multiplication matrix of $f \in \C[\x]$ with respect to a given basis $\OO = \{[b_1],\ldots, [b_N] \}$ of $\C[\x]/\I$. Then 
$$ \det(m_f - \lambda \mathbb{I}) = (-1)^N \prod_{\z \in \V(\I)} (\lambda - f(\z))^{\mu(\z)}$$
where $N = \dim \C[\x]/\I$, $\mathbb{I}$ is the identity matrix of size $N \times N$ and $\mu(\z)$ is the multiplicity of the root $\z$. Also, the row vector
$$\begin{bmatrix}
	b_1(\z)& \cdots & b_N(\z)
\end{bmatrix}$$
lies in the left eigenspace of the eigenvalue $f(\z)$ for each $\z \in \V(\I)$\footnote{Note that in general $\# \V(\I) \leq N$ where equality only holds if all solutions are simple.}.
\end{theorem}\begin{proof}
For the proof, we refer the reader to \cite[Chapter~4]{cox2}.
\end{proof}
Theorem \ref{theo:evals} implies that if we want to compute the coordinates of the solutions $z_1,\ldots, z_N$, we can construct the multiplication matrices $m_{x_1},\ldots, m_{x_N}$ corresponding to the coordinate functions and compute their eigenvalues. Another possibility is to use the eigenvectors \cite{stetter, cox2}. Note that, according to Theorem \ref{theo:evals}, the left eigenvectors do not depend on the choice of $f$. In fact, neither do the right ones. By their definition, it is not difficult to see that the multiplication maps must commute. They form a family of commuting matrices, so they must share common eigenspaces \cite{stetter}. We note here that when the set of eigenvectors spans $\C^N$ (that is, when all solutions are simple), the matrices $m_{x_1}, \ldots, m_{x_n}$ are simultaneously diagonalizable. We will give an example of the construction of the multiplication matrices of the coordinate functions in Section \ref{sec:normalform}. To work out this example, we will need the notion of a \textit{normal form}.

\begin{definition}[Normal form]
Let $\OO = \{[b_1], \ldots, [b_N]\}$ be a basis for $\C[\x]/\I$. Given any polynomial $g \in \C[\x]$, let 
$$ [g] = a_1 [b_1] + \ldots + a_N [b_N] = [a_1b_1 + \ldots + a_N b_N], \quad  a_i \in \C$$ be the unique representation of $[g]$ in the basis $\OO$. We say that $a_1b_1 + \ldots + a_N b_N$ is the \textup{normal form} of $g$ w.r.t. $\OO$. We denote $\overline{g}^{\OO} = a_1b_1 + \ldots + a_N b_N$.
\end{definition}
Note that for any $g \in \C[\x]$, if the basis elements $[b_i]$ are given by monomials: $[b_i] = [x^{\alpha_i}]$, we have that $S(\overline{g}^\OO) \subset \{\alpha_1,\ldots,\alpha_N \}$. In general $S(\overline{g}^\OO) \subset \bigcup_{i=1}^N S(b_i)$. The results in this section show that normal form algorithms for generic systems can be divided into two major parts. 
\begin{enumerate}
\item Compute the multiplication operators $m_{x_i}$, $i = 1, \ldots, n$.
\item Perform a simultaneous diagonalization of the $m_{x_i}$ to find the solutions or find the solutions via the eigenvectors of the $m_{x_i}$.
\end{enumerate}
In this paper, we focus on making improvements in the first step. The proposed algorithm will choose a basis $\OO$ for $\C[\x]/I$ such that the multiplication operators can be computed, heuristically, as accurately as possible.
\section{Motivation} \label{sec:motivation}
The normal form method presented in this paper is closely related to border basis algorithms and to multipolynomial Macaulay resultants. We briefly review some of their properties that are exploited or generalized in our algorithm. For more details on border bases we refer to \cite{mourrain1999new,elkadi_introduction_2007}, and for multipolynomial resultants to \cite{cox2}.

\subsection{Macaulay resultant matrices}
Consider the system of homogenized equations $f_1^h = \ldots =  f_n^h= 0$ coming from $I = \langle f_1, \ldots, f_n \rangle$. As discussed in Section \ref{sec:genericsystems}, the expected number of solutions in $\PP^n$ is finite and equal to B\'ezout's number. If we add a generic homogeneous equation $f_0^h = 0$ to the system, then generically the system has no solutions. The \textit{resultant} $\Res(f_0^h, \ldots, f_n^h)$ is a homogeneous polynomial in the coefficients of the $f_i^h$ that vanishes if and only if the system $f_0^h = \ldots = f_n^h = 0$ has a solution in $\PP^n$. A resultant matrix $M_0$ is a matrix such that $\det(M_0)$ is a nonzero multiple of the resultant polynomial. Several constructions of resultant matrices have been introduced \cite{elkadi_introduction_2007, cox2,cattani2005solving}. The one that is related in the most direct way to the algorithm presented in this paper is a generalization of the Sylvester matrix of two univariate polynomials to the multivariate case, also called the multipolynomial Macaulay resultant matrix \cite{cox2}. The rows in this matrix correspond to monomial multiples of the input equations, whereas its columns correspond to monomials, such that the coefficient of the polynomial corresponding to the $j$-th row coupled to the $i$-th monomial is the $(j,i)$ entry of the matrix. We denote this matrix by $M_0$. Let $f_0^h$ be a generic linear form and let $M_0$ be the Macaulay resultant matrix associated to $f_0^h, f_1^h, \ldots, f_n^h$. We view it as a block matrix
$$ M_0 = \begin{bmatrix}
M_{00} & M_{01} \\ M_{10} & M_{11}
\end{bmatrix},$$
such that the first block row $[M_{00} ~ M_{01}]$ contains the multiples of $f_0^h$ by the set of monomials
\begin{equation} \label{blockbasis1}
\OO_M = \{x_1^{\alpha_1}x_2^{\alpha_2}\cdots x_n^{\alpha_n} : 0\leq \alpha_i \leq d_i - 1, 1\leq i \leq n \}
\end{equation}
and the second block row contains the monomial multiples of the other $f_i^h$. Furthermore, we assume that the columns are ordered in such a way that the first block column $\begin{bmatrix}
M_{00} \\ M_{10}
\end{bmatrix}$ corresponds to the monomials in $\OO_M$. In \cite[Chapter 3]{cox2} it is shown that, with our genericity assumptions, we have that the Schur complement $m_{f_0} = M_{00} - M_{10}M_{11}^{-1}M_{01}$ represents multiplication by $f_0 = f_0^h(1,x_1, \ldots, x_n)$ in $\C[x]/I$ in the basis $\OO_M$. Macaulay \cite{macaulay1994algebraic} showed that generically, $M_{11}$ is invertible and hence it is possible to compute this Schur complement. One could, for instance, use $f_0^h = x_i$ to find $m_{x_i}$ in this way and find the solutions of $I$ by using the results in Section \ref{sec:multmtces}. This leads to a well known eigenvalue-eigenvector method for solving generic dense systems, based on resultants. However, when computations are performed in finite precision, the accuracy of the resulting matrix $m_{f_0}$ depends on the condition number of the inversion of $M_{11}$. That is, the `more invertible' $M_{11}$ is, the more accurate the operator $m_{f_0}$ can be obtained from this matrix, hence the more accurate we can compute its eigenstructure to find the points defined by $I$.  The algorithm proposed in this paper somehow chooses the partitioning of $M_0$ in an adaptive way, such that $M_{11}$ is well conditioned and the Schur complement still gives the multiplication map. 

\subsection{Border bases}
Groebner bases and Buchberger's algorithm to compute them provide an algorithmic, algebraic way to compute the solutions of a system of polynomial equations \cite{cox1,buchberger1998grobner,faugere1999new}. They can be used to compute normal forms in a basis for $\C[x]/I$ induced by a monomial order. A major drawback is that for large problems, Groebner bases are not feasible in finite precision, since the computations are unstable. Border bases have been developed as a generalization of Groebner bases to represent the quotient algebra $\C[x]/I$ \cite{mourrain_generalized_2005,mourrain_stable_2008,mourrain1999new}. With respect to Groebner bases, they enhance the numerical stability due to a more flexible choice of monomial bases for $\C[x]/I$ and they are also more robust (the border basis remains a basis for small perturbations of the coefficients) \cite{mourrain2007pythagore}. A border basis $\OO$ has the property that it is \textit{connected to 1}. This means that $1 \in \Span(\OO)$ and and, for any $g \in \Span(\OO)$ there are $g_1, \ldots, g_n \in \Span(\OO)$ such that 
$$ g = \sum_{i=1}^n x_i g_i.$$
The border basis criterion for normal form algorithms is given by the following theorem \cite{mourrain1999new}. 
\begin{theorem} \label{thm:BB}
Let $B = \Span(\OO) \subset \C[x]$ be such that $\OO$ is connected to 1. Let $N: B \cup \left(\bigcup_{i=1}^n x_i \cdot B \right) \rightarrow B$ be a $\C$-linear map such that it is the identity restricted to $B$. Let $I = \langle \ker N \rangle$ be the ideal generated by the kernel of $N$. Defining $M_i : B \rightarrow B : b \mapsto N(x_ib)$, the following properties are equivalent: 
\begin{enumerate}
\item $M_i \circ M_j = M_j \circ M_i$, 
\item $\C[x] = B \oplus I$.
\end{enumerate}
\end{theorem}
From $\C[\x] = \C[x]/I \oplus I$ it follows that when the $M_i$ from Theorem \ref{thm:BB} commute, $B \simeq \C[x]/I$ as $\C$-algebras and $M_i$ represents multiplication with $x_i$ modulo $I$ since $I = \langle\ker(N) \rangle$. Therefore, a basis $\OO$ must not be induced by a monomial order. It is sufficient that $\OO$ is connected to 1 and there is a map $N$ with the right properties: its kernel generates $I$ and the maps $M_i$ are pairwise commuting. Note that the basis $\OO_M$ from \eqref{blockbasis1} is connected to 1. We will show in Section \ref{sec:numexp} that this basis can still show bad numerical behaviour. In this paper we propose an algorithmic choice of basis that does not necessarily have the connected to 1 property. In border basis algorithms, the basis $\OO$ is fixed beforehand. This has the advantage that the algorithm can be adapted to this specific basis to reduce the computational cost. However, the choice of basis can influence the accuracy dramatically, as we will show in Section \ref{sec:numexp}. As specified in the following sections, the algorithm presented in this paper starts from a set of candidate monomials (which we will denote later by $S(M)_{<t}$) from which we will select the monomials in $\OO$. This set is quite large, and to choose the basis $\OO$ the algorithm uses numerical linear algebra techniques on a matrix called the \textit{Macaulay matrix}, very similar to Macaulay's resultant construction.

\section{Macaulay matrices} \label{sec:macmtx}
A \textit{Macaulay matrix} associated to the set of polynomials $\{f_1, \ldots, f_n \} \subset \C[\x]$ is a matrix over $\C$ in which each column corresponds to a monomial $\x^\alpha, \alpha \in \Ints^n$. Furthermore, such a Macaulay matrix has $n$ block rows, each of which corresponds to one of the polynomials in the set. The $j$-th row of the $i$-th block row is the vector representation of a polynomial $x^{\beta_{ij}}f_i \in \I, \beta_{ij} \in \Ints^n$ in the basis $\{\x^\alpha \}$ of monomials corresponding to the columns. For example, denote $R = \C[\x]$ and for an ideal $J \subset R$, we denote by $J_{\leq t}$ the elements in $J$ of degree $\leq t$. Let $d_i = \deg(f_i)$. For $t \geq \max_i d_i$, consider the linear map
\begin{alignat*}{2}
	 &\bigoplus_{i=1}^n R_{\leq t-d_i} \longrightarrow \I_{\leq t},  \\
	 &(a_1, \ldots, a_n) \longrightarrow a_1f_1+\cdots +a_nf_n.\end{alignat*} 
The transpose of the matrix representation of this map with respect to the standard monomial basis of $R_{\leq t}$ is a Macaulay matrix. We will call such a Macaulay matrix a \textit{dense Macaulay matrix}. We clarify this by means of an example.

\begin{example} \label{affmac}
Let $\I = \langle f_1,f_2 \rangle \subset \C[x_1,x_2]$ be generated by $f_1 = a +bx_1+cx_2$ and $f_2 = d + ex_1+fx_2 +gx_1^2+hx_1x_2 + jx_2^2$ with $a, \ldots, j \in \C$. It is clear that $\I_{\leq 2}$ is a subset of $R_{\leq 2}=\C[x_1,x_2]_{\leq 2}$ which is spanned as a $\C$-vector space by $1, x_1,x_2,x_1^2,x_1x_2,x_2^2$. Using this basis to represent elements of $\I_{\leq 2}$ and $R_{\leq 1} \oplus R_{\leq 0} = \Span(1,x_1,x_2) \oplus \Span(1)$ we get the transpose of the matrix

\[
  M=\kbordermatrix{%
      & 1  & x_1 & x_2 & x_1^2 & x_1x_2 & x_2^2  \\
    f_1 & a  & b & c &     &    &   \\
    x_1f_1&    & a &   & b   & c\\
    x_2f_1&    &   & a &     & b  & c \\
    f_2 & d  & e & f & g   & h  & j  
  }
\]
for the matrix representation of 
\begin{alignat*}{2}
	 &R_{\leq 1} \oplus R_{\leq 0} \longrightarrow \I_{\leq 2},  \\
	 &(a_1, a_2) \longrightarrow a_1f_1+a_2f_2.\end{alignat*}
The matrix $M$ is clearly a Macaulay matrix. 
\end{example}
By the \textup{support} of $M$, we mean the set of exponent vectors $$S(M) = \{\alpha \in \Ints^n : \x^\alpha \textup{ corresponds to a column of $M$} \}.$$
To describe the row content of $M$, we define the sets 
$$ \Sigma_i(M) = \{\beta_{ij} \in \Ints^n : \x^{\beta_{ij}} f_i \textup{ corresponds to a row of the $i$-th block row of $M$} \}.$$
The set $\Sigma_i$ is also called the set of \textup{shifts} of $f_i$. Note that, given the polynomials $f_i$, $M$ is defined up to row and column permutations by $S(M)$ and $\Sigma_i(M)$, $1\leq i \leq n$ and in order to be feasible, these sets must satisfy 
$$ S(x^{\beta_{ij}} f_i) \subset S(M), \forall \beta_{ij} \in \Sigma_i(M), 1\leq i \leq n.$$
\begin{example}
In the previous example, we have $S(M) = \{(0,0),(1,0),(0,1),(2,0),(1,1),(0,2)\}$, $\Sigma_1(M) = \{(0,0),(1,0),(0,1)\}$, $\Sigma_2(M) = \{(0,0)\}$. 
\end{example}
A Macaulay matrix of this type has a natural homogeneous interpretation. We show this by continuing the previous example. 
\begin{example} \label{hommac}
Homogenizing the equations we get $f_1^h = ax_0+bx_1+cx_2$ and $f_2^h = dx_0^2 + ex_0x_1+fx_0x_2 +gx_1^2+hx_1x_2 + jx_2^2$, where the superscript $h$ indicates the homogenization and it should not be confused with the coefficient $h \in \C$ of the monomial $x_1x_2$ in $f_2$. We denote $\overline{I} = \langle f_1^h,f_2^h \rangle \subset \C[x_0,x_1,x_2]$. Now, one can verify that $M$ is also the Macaulay matrix of $\{f_1^h,f_2^h\}$ with $$S^h(M) = \{(2,0,0),(1,1,0),(1,0,1),(0,2,0),(0,1,1),(0,0,2)\} \subset \Z^3,$$ $\Sigma_1^h(M) = \{(1,0,0),(0,1,0),(0,0,1) \}$, $\Sigma_2^h(M) = \{(0,0,0)\}$. It is clear how this can be generalized to any dense Macaulay matrix: if the associated map has image in $\I_{\leq t}$, homogenize the exponent vectors in $S(M)$ to degree $t$ in $\Ints^{n+1}$ and the exponent vectors in $\Sigma_i(M)$ to degree $t-d_i$. The associated linear map is given as follows. Denoting $R^h = \C[x_0,x_1,\ldots,x_n]$ and the degree $t$ part of a graded $R^h$-module $A$ by $A_t$ (the grading is induced by the standard grading on $R^h$), $M$ represents the map
\begin{alignat*}{2}
	 &\bigoplus_{i=1}^n R^h_{t-d_i} \longrightarrow \overline{I}_t,  \\
	 &(a_1, \ldots, a_n) \longrightarrow a_1f_1^h+ \cdots + a_nf_n^h.
	 \end{alignat*}
This map is surjective.
\end{example}
Macaulay matrices are used to give determinantal formulations of resultants \cite{cox2} and to solve systems of polynomial equations \cite{cox2,dreesen,bats}. They form a natural first step in reformulating the root finding problem as a linear algebra problem. The following theorem is straightforward \cite{dreesen}. 
\begin{theorem} \label{theo:null}
Let $S(M) = \{\alpha_1, \ldots, \alpha_l \}$ be the support of a Macaulay matrix $M$ of $\{f_1, \ldots, f_n \}$, where $\alpha_i$ corresponds to the $i$-th column of $M$. Let $\I = \langle f_1, \ldots, f_n \rangle$. The point $z \in \C^n$ satisfies $z \in \V(\I)$ if and only if the vector 
$$ v(z) = (x^{\alpha_1}(z), \ldots, x^{\alpha_l}(z))^\top$$
satisfies $Mv(z) = 0$.
\end{theorem}
It is clear that Theorem \ref{theo:null} generalizes to the projective interpretation of $M$. 
\begin{theorem} \label{theo:null2}
Let $S^h(M) = \{\alpha^h_1, \ldots, \alpha^h_l \}$ be the support of a (homogeneously interpreted) Macaulay matrix $M$ of $\{f_1^h, \ldots, f_n^h \}$, where $\alpha_i^h$ corresponds to the $i$-th column of $M$. Denote $\overline{I} = \langle f_1^h,\ldots,f_n^h \rangle$. The point $z^h \in \PP^n$ satisfies $z^h \in \V(\overline{I})$ if and only if the point 
$$ v(z^h) = (x^{\alpha_1^h}(z^h), \ldots, x^{\alpha_l^h}(z^h))^\top,$$
viewed as a point in $\PP^{l-1}$, satisfies $Mv(z^h) = 0$ (note that here $x = (x_0,x_1,\ldots, x_n)$ is short for an $n+1$-tuple). This condition is well defined, since $v(\lambda z^h) = \lambda^t v(z^h), \lambda \in \C \backslash \{0\}$ and $t = |\alpha_i^h|$.
\end{theorem}

Theorem \ref{theo:null2} implies that every point $z^h \in \V(\overline{I}) \subset \PP^n$ generates a direction $v(z^h)$ in the nullspace of $M$. We will now present a way to construct the dense Macaulay matrix such that its null space is spanned by these directions.
In the Macaulay matrix $M$ with support $S(M) = \{\alpha \in \Ints^n : |\alpha| \leq t \}$, the number of columns is $$\#S(M) =  \begin{pmatrix} t+n \\ n \end{pmatrix}.$$
Consider the shifts
$$ \Sigma_i(M) = \{ \beta \in \Ints^n : |\beta| \leq t-d_i \}.$$
It is clear that the resulting matrix $M$ is the dense Macaulay matrix of degree $t$. 
\begin{theorem} \label{theo:nullM}
Under our genericity assumptions, for $M$ constructed as above with $t \geq \sum_{i=1}^n d_i - n$, we have $\dim \nulll(M) = N$. Equivalently, for these values of $t$: $\#S(M) - N = \rank(M)$.
\end{theorem}
\begin{proof}
This result was known by Macaulay \cite{macaulay1994algebraic}. The degree $t = \sum_{i=1}^n d_i - n$ is called the \textit{degree of regularity} in \cite{bats,dreesen}.
\end{proof}
In fact, we have by construction that for the Macaulay matrix of degree $t$, $\dim \nulll(M)$ is the codimension of $\overline{I}_t$ in $R^h_t$, which is the dimension of $(R^h/\overline{I})_t$ as a $\C$-vector space. This is the Hilbert function of $\overline{I}$ evaluated at $t$ \cite{cox2,eisenbud_geometry_2005}. Since $\overline{I}$ defines points in $\PP^n$ by assumption, the Hilbert function becomes constant for large $t$ and Theorem \ref{theo:nullM} implies that this happens at $t = \sum_{i=1}^n d_i - n$. In the algorithm, we will also rely on the following theorem.
\begin{theorem} \label{theo:nbbasiselts}
For $t \geq \sum_{i=1}^n d_i - (n-1)$ we have that 
$$ \begin{pmatrix} t-1+n \\ n \end{pmatrix} \geq N,$$
with $N = \prod_{i=1}^n d_i$.
\end{theorem}
\begin{proof}
The number $ \begin{pmatrix} t-1+n \\ n \end{pmatrix}$ is the number of monomials of degree at most $t-1 = \sum_{i=1}^n d_i -n$. The number $N = \prod_{i=1}^n d_i$ is the number of monomials in the set $\{ \alpha \in \Ints^n : \alpha_i \leq d_i-1, i=1,\ldots, n \}$. The highest degree monomial in this set has degree $\sum_{i=1}^n d_i -n$. 
\end{proof}
It will become clear later that the properties of $M$ given in Theorem \ref{theo:nullM} and Theorem \ref{theo:nbbasiselts} are exactly the properties we need in our algorithm. We also want $M$ to be as small as possible to reduce memory use and computational effort. We therefore set $t = \sum_{i=1}^n d_i -(n-1)$.
\section{Normal form computation using the Macaulay matrix} \label{sec:normalform}
In this section, we propose a new normal form algorithm for computing the $m_{x_i}$ for a generic dense system as described in Section \ref{sec:genericsystems}.
\subsection{An example}
We introduce the ideas of our algorithm by a simple example. Consider the ideal $\I = \langle f_1, f_2 \rangle \subset \C[x_1,x_2]$ given by $f_1(x_1,x_2) = x_1^2+x_2^2 -2 = 0,f_2(x_1,x_2) = 3x_1^2 -x_2^2 -2 = 0$.
We will use linear combinations of $f_1,xf_1,yf_1,f_2,xf_2,yf_2$ to find the normal forms. The variety $\V(\I) = \{(-1,-1),(-1,1),(1,-1),(1,1)\}$ is 0-dimensional and the system satisfies the genericity assumptions. A possible basis for $\C[x_1,x_2]/\I$ is $\OO = \{[1], [x_1],[x_2],[x_1x_2]\}$. We construct the dense Macaulay matrix $M$ of degree $t = \sum_{i=1}^2d_i - (n-1) = 3$ as presented in Section \ref{sec:macmtx}, ordering the columns such that these monomials correspond to the last four columns: 
\[
  M=\kbordermatrix{%
      &  x_1^3 & x_1^2x_2 & x_1x_2^2 & x_2^3 &x_1^2 & x_2^2 && \underline{1}  & \underline{x_1} & \underline{x_2} &  \underline{x_1x_2}   \\
    f_1 &   &  &  &  & 1  & 1 &\vrule& -2   \\
    x_1f_1 & 1 &  & 1 &      &  &&\vrule& & -2 &  &  \\
    x_2f_1 &  & 1  &  & 1    &   &&\vrule&   &  &-2&  \\
    f_2 &   &  &  &&  3  & -1 &\vrule&-2 \\
    x_1f_2 & 3  &  & -1 &   & &   &\vrule& & -2 &  &  \\
    x_2f_2 &  & 3  &  & -1   &   & &\vrule&   &  & -2 & 
  }.
\]
To construct the multiplication maps $m_{x_1}$ and $m_{x_2}$ with respect to $\OO$, we need to calculate the normal forms of $x_1^2, x_1^2x_2, x_2^2, x_1x_2^2$ in $\OO$. All of these monomials appear in the left block column of $M$. Inverting this column block and applying it from the left to $M$ gives
\[
  \tilde{M}=\kbordermatrix{%
      &  x_1^3 & x_1^2x_2 & x_1x_2^2 & x_2^3 &x_1^2 & x_2^2 && \underline{1}  & \underline{x_1} & \underline{x_2} &  \underline{x_1x_2}   \\
     x_1^3-x_1 & 1 &  &  &  &  &  &\vrule& & -1  &  &\\
     x_1^2x_2-x_2&   & 1&  &  &  &  &\vrule&&  & -1 &  \\
     x_1x_2^2-x_1 &   &  & 1&  &  &  &\vrule&  &-1&  &  \\
     x_2^3-x_2 &   &  &  & 1&  &  &\vrule&  &  &-1&\\
     x_1^2-1&   &  &  &  &1 &  &\vrule& -1 &&  &  \\
     x_2^2-1&   &  &  &  &  &1 &\vrule& -1 &  && 
  }.
\]
Note that the left block was square because of the properties of the dense Macaulay matrix. The rows of $\tilde{M}$ are linear combinations of the rows of $M$, so they represent polynomials in $\I$. Hence, for example, $[x_1^2 - 1] = [0]$ modulo $\I$ and the normal form of $x_1^2$ is $1$. Using the information in $\tilde{M}$ we can construct $m_{x_1}$ and $m_{x_2}$. This gives
\[
  m_{x_1}=\kbordermatrix{%
   &{[x_1]\cdot[1]}&{[x_1]\cdot[x_1]}&{[x_1]\cdot[x_2]}&{[x_1]\cdot[x_1x_2]} \\
    {[1]} & 0  & 1  & 0 & 0 \\
    {[x_1]} & 1 & 0 &  0  & 0\\
    {[x_2]} & 0  & 0  & 0  & 1\\
    {[x_1x_2]} & 0  & 0 & 1 & 0\\
  },
\]
\[
  m_{x_2}=\kbordermatrix{%
   &{[x_2]\cdot[1]}&{[x_2]\cdot[x_1]}&{[x_2]\cdot[x_2]}&{[x_2]\cdot[x_1x_2]} \\
    {[1]} & 0  & 0  & 1 & 0 \\
    {[x_1]} & 0 & 0 &  0  & 1\\
    {[x_2]} & 1  & 0  & 0  & 0\\
    {[x_1x_2]} & 0  & 1 & 0 & 0\\
  }.
\]
Note that the first and the third column of $m_{x_1}$ are trivial and so are the first and the second column of $m_{x_2}$. The other columns can be read off $\tilde{M}$ directly. The eigenvalues of $m_{x_i}$ coincide with the $i$-th coordinates of the points in $\V(\I)$. 

\subsection{The monomial basis}
When choosing the basis $\OO$, we must take into account that $\OO$ cannot contain monomials of degree $t$ (3 in the previous example). Otherwise, multiplying with $x_1$ or $x_2$ gives a monomial that is not in $S(M)$. Secondly, it must be such that the resulting system is solvable. In the generic case, there is always such a choice. We consider the Macaulay matrix $M$ of degree $t = \sum_{i=1}^n d_i - (n-1)$. By $S(M)_t$ we denote the monomials in $S(M)$ of degree $t = \sum_{i=1}^n d_i - n + 1$ and by $S(M)_{<t}$ the remaining monomials. We order the columns of the Macaulay matrix in such a way that 
$$M = \begin{bmatrix} M_b & M_i & B \end{bmatrix}$$
where $M_b$ are the columns corresponding to $S(M)_t$, $B$ contains the columns corresponding to $\OO$ and $M_i$ corresponds to $S(M)_{< t} \backslash \OO$. When the polynomials $f_1, \ldots, f_n$ are generic, the set of monomials
\begin{equation} \label{blockbasis}
\OO_M = \{x_1^{\alpha_1}x_2^{\alpha_2}\cdots x_n^{\alpha_n} : 0\leq \alpha_i \leq d_i - 1, 1\leq i \leq n \} \subset S(M)_{\leq t}
\end{equation}
is a basis for $\C[\x] /\I$ \cite{cox2} (this is exactly the basis used in the Macaulay resultant construction, as introduced in Section \ref{sec:motivation}). This means that every monomial in $S(M) \backslash \OO_M$ has a unique normal form in $\OO_M$. In other words, there is a unique polynomial of the form 
\begin{equation} \label{eq:NF}
g_\alpha = x^\alpha - \sum_{b \in \OO_M} c_{\alpha b} b \in \I
\end{equation}
with $c_{\alpha b} \in \C$, for each $\alpha \in S(M) \backslash \OO_M$. Also,  it follows from Property (iii) in \cite[Chapter 1, p.46]{cattani2005solving}  and from our assumptions that for $t \geq \sum_{i=1}^n d_i-(n-1)$, the rows of $M$ span $I_{\leq t}$ linearly. Now, since $g_\alpha \in I_{\leq t}$ and the rows of $M$ span $I_{\leq t}$, every $g_\alpha$ is a linear combination of the rows of $M$. The number of polynomials $g_\alpha$ is $r = \rank(M)$, so we can apply a square matrix to the left of $M$ to transform $M$ into 
\begin{equation} \label{canonform}
\begin{bmatrix}
1 &  &  &  & -c_{\alpha_1 b_1} & \cdots & -c_{\alpha_1 b_N}\\
  & 1&  &  & -c_{\alpha_2 b_1} & \cdots & -c_{\alpha_2 b_N}\\
  &  &\ddots&  & \vdots &  & \vdots \\
  &  &  & 1& -c_{\alpha_r b_1} & \cdots& -c_{\alpha_r b_N} \\ \hline 
0 &0 &\cdots& 0 &0 & \cdots& 0   \\
\vdots &&&&&& \vdots\\
0 &0 &\cdots& 0 &0 & \cdots& 0   \\
\end{bmatrix}.
\end{equation}
The block row of zeros is introduced by the syzygies in the rows of $M$\footnote{This occurs only for $n \geq 3$, not in the example given here.} \cite{eisenbud_geometry_2005}. This proves the following theorem. 
\begin{theorem} 
The matrix $M_b$ is of full column rank under our assumptions, and there is at least one possible choice of $\OO$ for which $\begin{bmatrix} M_b & M_i \end{bmatrix}$ is of rank $r$.
\end{theorem}
However, there are many more choices for $\OO$ than the `block basis' from \eqref{blockbasis}. From a numerical point of view, it turns out this is crucial to find the normal forms with high accuracy. The idea is simple: we choose $\OO$ in such a way that $\begin{bmatrix} M_b & M_i \end{bmatrix}$ is `as invertible as possible', i.e., it has a small condition number. \\
 
\subsection{Algorithm}
We propose to make the choice of basis $\OO$ by using a QR factorization with optimal column pivoting on (part of) the Macaulay matrix. This is a well known numerical linear algebra algorithm to compute an upper triangularization of a column permuted version of a matrix such that the diagonal elements are, heuristically, as large as possible (in absolute value). See for instance \cite{golub1965numerical}. This leads to Algorithm \ref{alg:dense} for the generic dense case. 
\begin{algorithm}
\caption{Multiplication maps of a dense system}\label{alg:dense}
\begin{algorithmic}[1]
\Procedure{MultMatrices}{$f_1,\ldots,f_n$}
\State $M \gets \textup{dense Macaulay matrix of degree } \sum_{i=1}^n d_i - (n-1)$ \label{constmac}
\State $M \gets \begin{bmatrix} M_b & M_*\end{bmatrix} = MP_c$ \label{permcols}
\State $M_b = Q_b R_b$ \label{qrborder}
\State $M \gets Q_b^* M = \begin{bmatrix} \hat{R}_b & Z \\ 0 & \hat{M}_*  \end{bmatrix}$ \label{triangborder}
\State $ \hat{M}_* P_i = Q_i R_i$ \label{qrinterior}
\State $M \gets M \begin{bmatrix} \mathbb{I} & 0 \\ 0 & P_i \end{bmatrix} $ \label{perminterior}
\State $M \gets \begin{bmatrix} \mathbb{I} & \\ & Q_i^* \end{bmatrix} M = \begin{bmatrix} \hat{R}_b & Z P_i \\ 0 & \tilde{R}_i \\ 0& 0 \end{bmatrix} = \begin{bmatrix} \hat{R}_b & \hat{Z}_1 & \hat{Z}_2 \\ 0 & \hat{R}_i & \hat{Z}_3 \\ 0 & 0 & 0 \end{bmatrix}$ \label{triangrest}
\State $M \gets \begin{bmatrix} \hat{R}_b & \hat{Z}_1 & \hat{Z}_2 \\ 0 & \hat{R}_i & \hat{Z}_3 \end{bmatrix}$  \label{dropsyzs}
\State $C \gets -\begin{bmatrix} \hat{R}_b & \hat{Z}_1 \\ 0 & \hat{R}_i \end{bmatrix}^{-1} \begin{bmatrix} \hat{Z}_2 \\ \hat{Z}_3 \end{bmatrix}$ \label{solve}
\For{$i=1,\ldots,n$}
\State $\textup{Construct $m_{x_i}$ using the normal forms in $C$.}$
\EndFor
\State \textbf{return} $m_{x_1},\ldots, m_{x_n}$
\EndProcedure
\end{algorithmic}
\end{algorithm}
We briefly go through the different steps of the algorithm.
\begin{itemize}
\item Step \ref{constmac} is obvious. In step \ref{permcols}, we re-arrange the columns of $M$ such that $M_b$ contains the columns corresponding to $S(M)_t$ and $M_*$ contains all of the other columns. The order within the block columns is of no importance. We represented this in Algorithm \ref{alg:dense} by a column permutation matrix $P_c$. At this point, we do not split $M_*$ into $M_i$ and $B$ as before. The actual choice of basis is made in step \ref{qrinterior}.
\item In step \ref{qrborder}, we compute a QR factorization of the border block $M_b$, to make this block column upper triangular in step \ref{triangborder}. The matrix $\hat{R}_b$ is the square upper triangular part of $R_b$, it is a nonsingular matrix. At this point, the lower block row represents polynomials in $\I$ of degree $\leq t-1$.
\item Step \ref{qrinterior} is essential. We perform a QR factorization \textit{with optimal column pivoting} to the full lower right block. That is, we do not compute a QR factorization of $\hat{M}_*$, but of a column permuted version $\hat{M}_*P_i$, where $P_i$ is a column permutation matrix. The column permutation is such that it heuristically selects the `linearly most independent' columns first. In step \ref{perminterior} we apply the corresponding permutation to the entire matrix $M$ and in step \ref{triangrest} we make the entire matrix upper triangular ($\tilde{R}_i$ is the upper non-zero block and $\hat{R}_i$ is the square upper triangular part of $\tilde{R}_i$). We split the right block column into two block columns such that $\hat{R}_i$ is square. Under our assumptions, $\hat{R}_i$ is of full rank. Note that in the result, columns are still associated to monomials and the rows are polynomials in $\I$. With increasing row index, the support of these polynomials is contained in a shrinking subset of $S(M)$. Note that in this step, the syzygies introduce a block row of zeros in $M$. We drop this block row of zeros in step \ref{dropsyzs}. Denoting $r = \rank(M)$, the remaining matrix $M$ is of size $ r \times (N+r)$ by the results of Section \ref{sec:macmtx}.
\item In step \ref{solve}, we take out the leftmost $r \times r$ upper triangular block and apply its inverse to the right most $r \times N$ part with opposite sign to find the normal forms of all the monomials corresponding to the first $r$ columns. Of course, we do not calculate the inverse, but apply backsubstitution instead. It is the condition number of this inversion that is controlled by the optimal column pivoting in step \ref{qrinterior}.
\end{itemize}

\subsection{Connection with resultants and border bases}
To show the relation with the Macaulay resultant construction we assume that $\OO_M$ from \eqref{blockbasis} is a basis for $\C[x]/I$. Note that the row space of $M$ is equal to the row space of the second block row of $M_0$, denoted by $[M_{10} ~ M_{11}]$ in Section \ref{sec:motivation}. It is isomorphic to the subvector space $\overline{I}_t$ (the degree $t$ part) of $\overline{I}$. The resultant construction uses monomial multiples of the input equations that generically generate this subspace. Our construction uses all of the possible monomial multiples, which leads to more computational effort since we need to perform a row compression (step \ref{triangrest}), but we observe numerically that the computed basis of the row space after this compression (first two block rows of $M$ in step \ref{triangrest}) has larger singular values. \\

Suppose that in steps \ref{qrinterior}, \ref{perminterior} we choose the basis $\OO_M$ from \eqref{blockbasis} instead of performing the QR factorization with pivoting. We can apply an invertible transformation to $M$ in step \ref{dropsyzs} on the left such that it becomes equal to $[M_{10} ~ M_{11}]$ up to column permutation. In fact, by construction, $M_{11}$ corresponds to the square, invertible, upper triangular part $\begin{bmatrix} \hat{R}_b & \hat{Z}_1 \\ 0 & \hat{R}_i \end{bmatrix}$ of $M$, since the basis monomials correspond to $M_{10}$ and $\begin{bmatrix}
\hat{Z}_2 \\ \hat{Z}_3
\end{bmatrix}$. Choosing another basis $\OO$ yields another matrix $M_{11}$. As long as it remains invertible and the multiples $\OO \cdot f_0$ are supported in $S(M)$, $m_{f_0}$ can be computed as the Schur complement given in Section \ref{sec:motivation} by concatenating the shifts of $f_0$ (corresponding to $[M_{00} ~ M_{01}]$) to $M$ in step \ref{dropsyzs} and rearranging the blocks such that the lower right one is occupied by $\begin{bmatrix} \hat{R}_b & \hat{Z}_1 \\ 0 & \hat{R}_i \end{bmatrix}$. In fact, for any column permutation that gives a full rank $\begin{bmatrix} \hat{R}_b & \hat{Z}_1 \\ 0 & \hat{R}_i \end{bmatrix}$, we can bring $M$ into the form \eqref{canonform} and it is clear that
$$\C^{\#S(M)} \simeq R^h_t = \overline{I}_t \oplus \Span(\OO^h)$$
with $R^h = \C[x_0, x_1, \ldots, x_n]$ and $\OO^h$ contains the homogenized monomials in $\OO$ (they are homogenized to degree $t$). Also, $R^h_t = \overline{I}_t \oplus (R^h/\overline{I})_t$ so 
$$\Span(\OO) \simeq \Span(\OO^h) \simeq \Span(\OO_M) \simeq (R/\overline{I})_t \simeq \C[\x]/I.$$ These are isomorphisms of $\C$-vector spaces. The first one is given by homogenization and the last isomorphism follows from the genericity assumptions and the fact that the Hilbert polynomial stabilizes at regularity \cite{eisenbud_geometry_2005}. The isomorphism $\Span(\OO) \simeq \Span(\OO_M)$ is given by an invertible `change of basis' matrix constructed as follows. The basis transformation $\OO_M \rightarrow \OO$ is a matrix with columns equal to the normal forms of $\OO_M$ in $\OO$. Call this matrix $T$. Then the multiplication maps $m_{x_i}$ in $\OO$ give multiplication maps $m_{x_i}'$ in $\OO_M$, given by $m_{x_i}' = T^{-1}m_{x_i}T$. This transformation makes $\Span(\OO) \simeq \Span(\OO_M) \simeq \C[x]/I$ isomorphisms of $\C$-algebras.\\

To show the connection with border bases, we will also work with $\OO_M$ for simplicity. Any other border basis will do. Suppose we choose the basis $\OO_M$ in steps \ref{qrinterior}, \ref{perminterior} to compute $m_{x_1}', \ldots, m_{x_n}'$ in this basis. Note that this basis is connected to 1. Set $B = \Span(\OO_M)$ and consider the $\C$-linear map 
\begin{alignat*}{2}
	 N:~& B \cup \left(\bigcup_{i=1}^n x_i \cdot B \right) \longrightarrow B,  \\
	 &b \longmapsto \begin{cases}
	 b & b \in B \\
	 m_b'(1) & b \notin B
	 \end{cases}
	 \end{alignat*}
where $m_b' = b(m_{x_1}', \ldots, m_{x_n}')$. We show that $I = \langle \ker(N) \rangle$. Since $m_{x_i}'$ represents multiplication with $x_i$ modulo $I$, $m_b'$ represents multiplication with $b$ modulo $I$ and we have that $\ker(N) \subset I$ and hence $\langle \ker(N) \rangle \subset I$. Let $K = (\bigcup_{i=1}^n x_i \cdot \OO_M) \backslash \OO_M$ and let $g_\alpha$ be the polynomial \eqref{eq:NF} for every $x^\alpha \in K$. Note that $\ker(N) = \Span(g_\alpha, \alpha \in K) \subset I$. Any polynomial $f \in \C[x]$ can be written as $f = \sum_{\alpha \in K} c_\alpha g_\alpha + \overline{f}^{\OO_M}$ with $c_\alpha \in \C[x]$ using a division algorithm as described in Chapter 4 of \cite{cattani2005solving}. Since $\OO_M$ is a (border) basis for $\C[x]/I$, $\overline{f}^{\OO_M} = 0$ when $f \in I$. Therefore $I \subset \langle g_\alpha, \alpha \in K \rangle = \langle \ker(N) \rangle \subset I$ so $I = \langle g_\alpha, \alpha \in K \rangle = \langle \ker(N) \rangle$. It follows from Theorem \ref{thm:BB} that $M_i(b) = N(x_i b) = m_{x_i}'(b)$ represents multiplication with $x_i$ in the basis $\OO_M$ for $\C[x]/I$ and the $m_{x_i}'$ commute. For any other basis $\OO$ with transformation matrix $T : \OO_M \rightarrow \OO$, commutativity of the resulting $m_{x_i}$ follows from the relation $m_{x_i} = T m_{x_i}' T^{-1}$. 
\section{Numerical experiments} \label{sec:numexp}
In this section, we use Algorithm \ref{alg:dense} for some numerical experiments and compare it to Bertini \cite{bates2013numerically,bates2016bertinilab} and PHClab \cite{guan2008phclab}. 

\subsection{Evaluating a polynomial function on $\V(\I)$}
Theorem \ref{theo:evals} implies that we can evaluate a function $f \in \C[\x]$ on $\V(\I)$ by calculating the eigenvalues of $m_f= f(m_{x_1}, \ldots, m_{x_n})$. Note that this expression for $m_f$ is well defined because of the commutativity of the $m_{x_i}$. Algorithm \ref{alg:dense} can be used if $\I$ satisfies the assumptions made in this paper. As a test of correctness, we have evaluated the quadric $f(x_1,x_2) = -(x_1^2+x^2_2) + 0.1 xy + 15$ on the variety defined by two bivariate polynomials of degree 7 and 6, shown in Figure \ref{fevalpicture}. For the computed multiplication matrices, we compute 
$$ \frac{\lVert m_{x_1}m_{x_2} - m_{x_2}m_{x_1} \rVert_2}{\lVert m_{x_1}m_{x_2} \rVert_2} = 5.5552 \cdot 10^{-13}.$$
This shows that the multiplication matrices commute (up to 13 digits of accuracy).

\begin{figure}
\input{fevalflat.tex}
\includegraphics[width= 8cm]{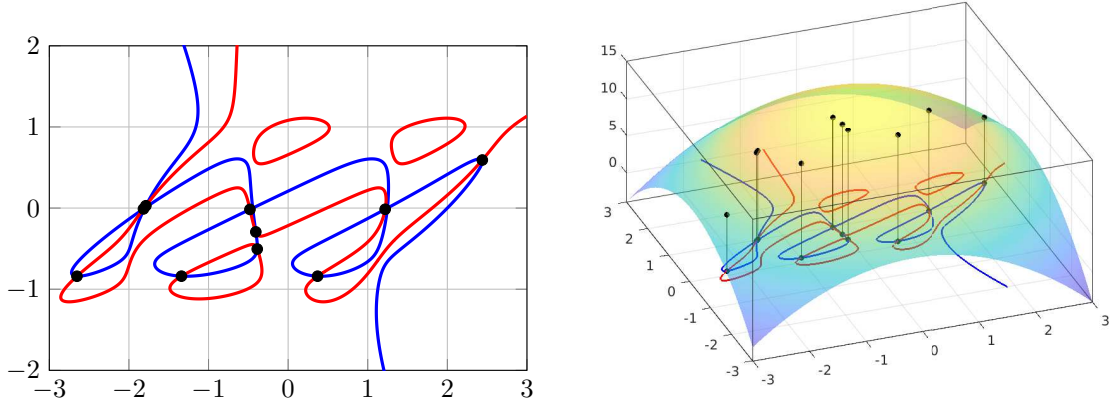}
\caption{Left: zero level lines in $\R^2$ of two bivariate polynomials of degree 7 (\ref{blue}) and 6 (\ref{red}) together with the real solutions (\ref{sols}). Right: The surface $f(x_1,x_2) = -(x_1^2+x^2_2) + 0.1 xy + 15$ and the real eigenvalues of $m_f$.}
\label{fevalpicture}
\end{figure}
\subsection{Solving generic systems}
We now use the obtained multiplication maps to compute the solutions $\V(\I)$ of square systems of polynomial equations in the following way. We perform a simultaneous diagonalization of the identity matrix together with the $n$ multiplication maps $m_{x_i}$. For this, we use the method \texttt{cpd\_gevd} in Tensorlab \cite{vervliet2016tensorlab,LRA1993,dL2006}. We compare the results (accuracy and computation time) with the homotopy solvers BertiniLab \cite{bates2016bertinilab} and PHClab \cite{guan2008phclab}. To obtain the results, we used Matlab and we generated generic polynomials $f$ in the following way. We fix a Newton polytope $P$ of $f$ and to every point in $P \cap \Ints^n$ we assign a real number drawn from a normal distribution with $\mu = 0$ and $\sigma = 1$ (using the \texttt{randn} command in Matlab). These numbers are the coefficients of the monomials in $S(f)$. To measure the accuracy of the resulting multiplication matrices, we calculate the condition number of the matrix inverted in step \ref{solve} of Algorithm \ref{alg:dense}. The accuracy of a solution $z$ of a square system $f_1 = \ldots = f_n = 0$ is measured by the \textit{residual}.
\begin{definition}
Given a square system of polynomial equations $f_1 = \ldots = f_n = 0$ with $f_1, \ldots,f_n \in \C[\x]$ and a point $z \in \C^n$. The \textup{residual} $r$ of $z$ is defined as 
$$ r_i = \frac{|f_i(z)|}{f_{i,\textup{abs}}(z_{\textup{abs}}) +1}, \qquad r = \frac{1}{n} \sum_{i=1}^n r_i,$$
where $|\cdot|$ denotes the absolute value, $f_{i,\textup{abs}}$ is $f_i$ where the coefficients $c_{\alpha,i}$ of $f_i$ are replaced by their absolute values and $z_{\textup{abs}}$ is the point in $\C^n$ obtained by taking the absolute values of all the components of $z$. 
\end{definition}
The term $+1$ in the denominator of the $r_i$ makes it clear that we are using a mixed relative and absolute criterion, to take into account the possibility that $f_{i,\textup{abs}}(z_{\textup{abs}})=0$.\\

We first investigate the influence of the automated choice of basis made in our algorithm. We compare it to the fixed choice of the block basis given in \eqref{blockbasis}. This is the basis that is used (implicitly) in root finding using $u$-resultants \cite[Chapter 3]{cox2}. We first check that it is not just the block basis itself that comes out of our algorithm. We generated two random dense polynomials $f_1,f_2 \in \C[x_1,x_2]$ of degree $d_1 = d_2 = 10$. The support of the associated dense Macaulay matrix $M$ is all monomials of degree up to $d_1+d_2-1 = 19$. The basis $\OO$ should count 100 elements (Theorem \ref{bezout}). Figure \ref{fig:basischoice} shows that, indeed, the choice of basis is significantly different. Note also that the resulting basis does not have the connected to 1 property. \\
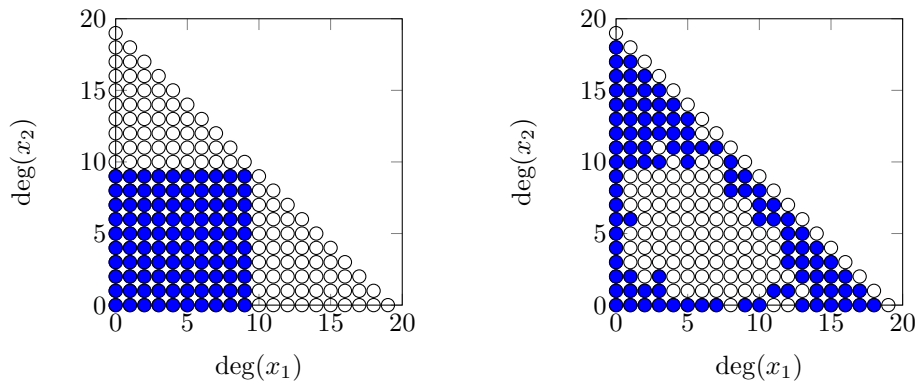
\begin{figure}[H]
\centering
\input{basischoice.tex}
\caption{Left: the block basis $\OO$ given in \eqref{blockbasis}. Right: the basis $\OO$ chosen by Algorithm \ref{alg:dense}. Black circles indicate the support $S(M)$ of the Macaulay matrix. }
\label{fig:basischoice}
\end{figure}
We now check the accuracy of the multiplication matrices by computing the condition number of the coefficient matrix inverted in step \ref{solve}. For a condition number of order $10^l$, we expect to loose $l$ accurate digits w.r.t.~the machine precision. Figure \ref{fig:CNR2} shows the results for bivariate systems of increasing degree\footnote{By degree $d$ we mean here that both polynomials $f_1$ and $f_2$ are generic of degree $d$.} up to 20. By using the QR decomposition with optimal column pivoting the condition number is controlled and it gets no larger than $\pm 10^4$. With our machine precision of order $10^{-16}$ (double precision), this means that the forward error on the multiplication matrices is of order $10^{-12}$. For the same set of generic bivariate systems of degree 1 up to 20 we also calculated the maximal residual of all of the calculated solutions. This is shown in the right part of Figure \ref{fig:CNR2}. One can expect that more accurate multiplication maps lead to more accurate solutions, which is confirmed by the figure. For degrees higher than 15, the solutions obtained using the block basis no longer made sense. The results are avaraged out over 20 experiments. These results clearly show that a numerically justified choice of basis is crucial for the feasibility of normal form algorithms to compute multiplication matrices. 
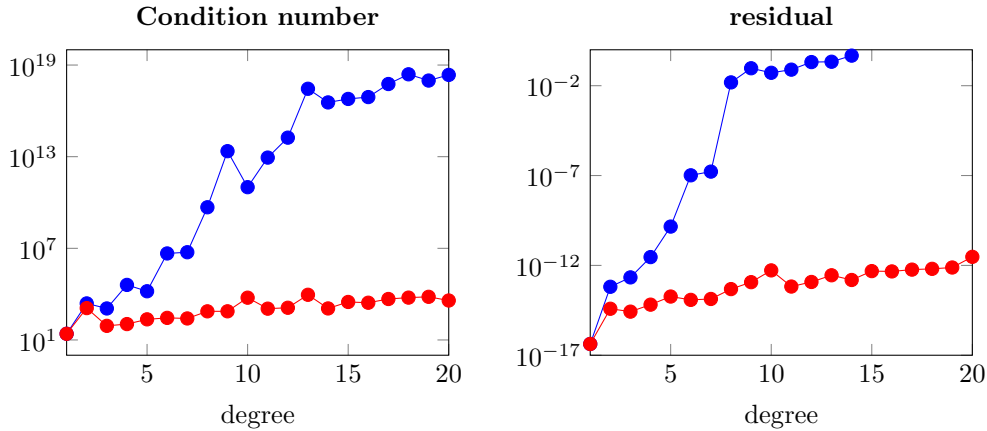
\begin{figure}[H]
\centering
\input{condition_numbers_2.tex}
\caption{Left: condition number for the computation of the multiplication matrices with block basis (\ref{block_res_2}) and smart choice of basis (\ref{auto_res_2}) for bivariate systems of increasing degree. Right: Maximal residual with the block basis (\ref{block_res_2}) and the QR choice of basis (\ref{auto_res_2}) for the same systems. }
\label{fig:CNR2}
\end{figure}
In the following, we only work with the automated choice of basis. Some results for dense systems with more variables are shown in Figure \ref{fig:densemorevar}. The figure shows that even for large systems, all solutions are found with a small residual. For example, in the case $n = 3$ with degree $21$, there are 9261 solutions in $\C^3$, all found with a residual smaller than $10^{-10}$. We also note that the residual would drop to machine precision after one `refining' iteration of Newton's method.\\

As for the computation time, the figure shows that the method is very sensitive to the number of variables (it suffers from the `curse of dimensionality'). The asymptotic complexity is $O(d^{3n})$, where $d$ is the degree. Intuitively, we find the coordinates of the $d^n$ solutions as eigenvalues and the cost of the algorithm is the number of eigenvalues cubed. \\

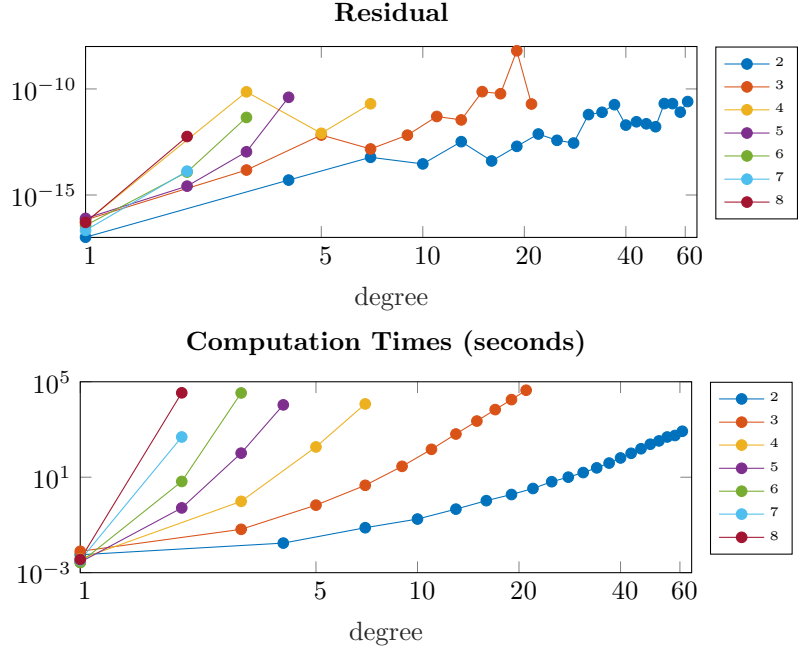
\begin{figure}
\centering
\input{allresiduals.tex}
\input{alltimes.tex}
\caption{Maximal residual and computation times for systems of increasing degree with $n = 2, \ldots, 8$.}
\label{fig:densemorevar}
\end{figure}

We compare our method to the Matlab interfaces of the homotopy continuation packages Bertini \cite{bates2016bertinilab} and PHCpack \cite{guan2008phclab}\footnote{We used default settings for both solvers.}. The results are shown in Figure \ref{fig:comparison}. The figure confirms that the complexity of our method grows drastically with $n$. For $n=2$, however, we are slightly faster for degrees at least up to 25. In all figures, the residuals of our computed solutions are slightly bigger than the ones from the homotopy methods. This is because these methods intrinsically make use of Newton-Raphson refinement. One Newton sweep over our solutions would lead to a residual of order machine precision as well, because of the quadratic convergence property. An important remark is that continuation methods do not return \textit{all} solutions in all cases. The methods might give up on certain paths along the way if the algorithm decides that the path seems to be diverging to infinity or if it crosses an ill-conditioned region. For $n = 2$ and degree 20, PHClab returns 398 solutions (2 solutions are lost) within slightly less than 4 seconds. For $n = 2$, degree 40, it takes 57 seconds to find 1575 out of the 1600 solutions. Using Bertini with double precision arithmetic \cite{bates2013numerically}, we find all solutions for $n=2$, degree 20 within 12 seconds and 1587 out of 1600 solutions for $n=2$, degree 40 within 350 seconds. 

\begin{figure}
\centering
\input{phcbrt23.tex}
\input{phcbrt45.tex}
\caption{Comparison of the results for PHClab (\ref{phc}), BertiniLab (\ref{brt}) and our method (\ref{qr}) for $n = 2,3,4,5$.}
\label{fig:comparison}
\end{figure}
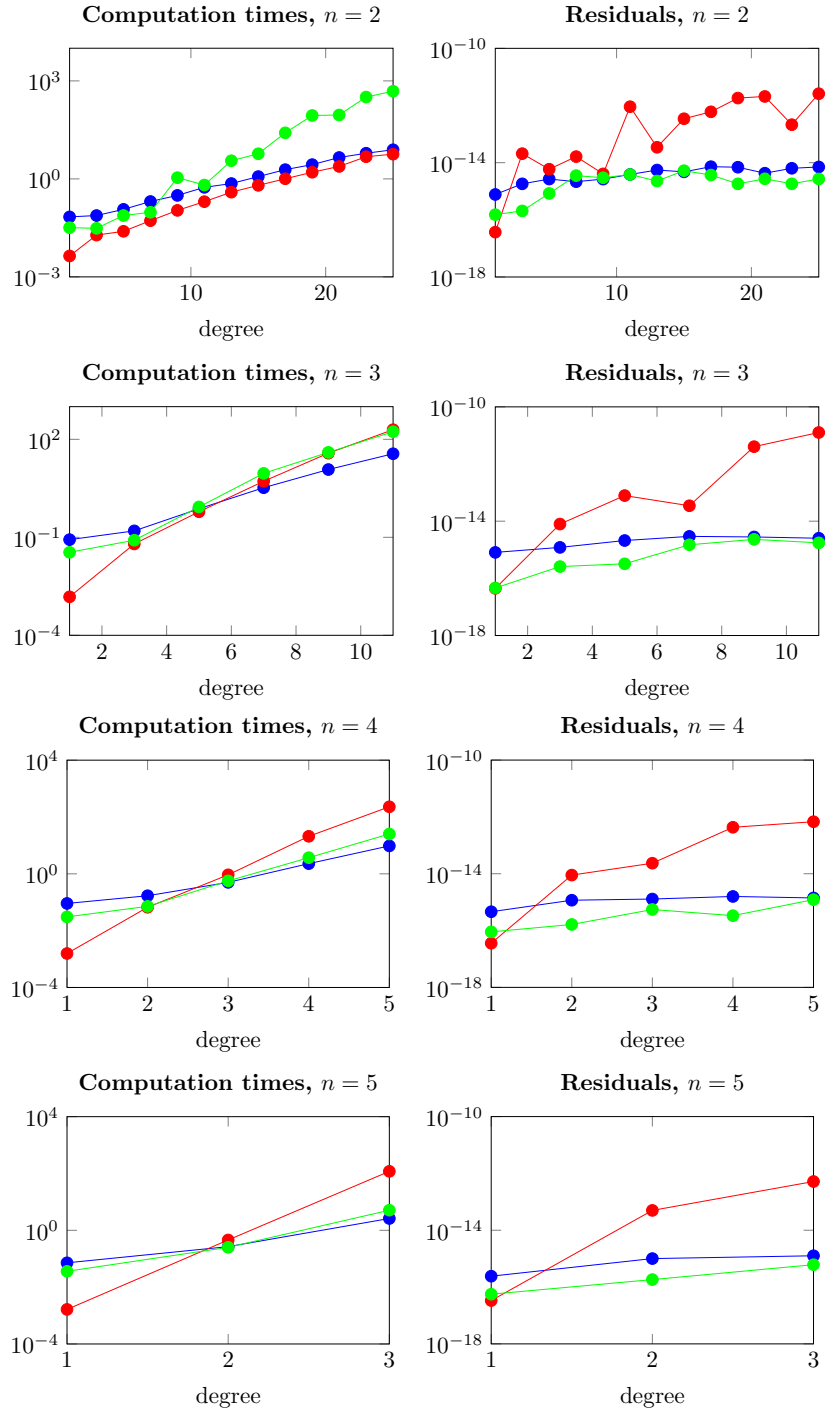

\section{Conclusion and future work}
We have presented a first normal form algorithm for zero-dimensional ideals coming from square polynomial systems that makes an automated, numerically justified choice of monomial basis for $\C[\x]/\I$ under certain genericity assumptions on $\I$. Our numerical experiments show that this choice of basis makes it possible to perform the normal form computation in finite precision, while it can go terribly wrong by manually choosing a basis. Some ideas for future work are: 
\begin{itemize}
\item Relaxing the genericity assumptions. What if the polynomials $f_1, \ldots, f_n$ are sparse? 
\item Solutions at infinity lead to linear dependencies in the columns of $M_b$, but it also causes the dimension of $\C[\x]/\I$ to drop. This can be incorporated in the algorithm.
\item For multiple solutions of a square polynomial system, the canonical polyadic decomposition does not work. The coupling between the different coordinates can be made by using the left eigenvectors of the multiplication maps. 
\item The implementation is done in Matlab and a lot of computation time is spent on the construction of the Macaulay matrix $M$. We believe that an implementation in Julia, C(++), Fortran, \ldots~could be a significant improvement.
\item Taking all possible monomial multiples of the input equations to construct $M$ leads to a number of redundant rows. This number becomes large very quickly when the number of variables increases. To reduce the computational cost and to enhance the performance for a larger number of variables, only a selection of the monomial multiples can be used. In doing so, a trade-off between speed and accuracy should be taken into account. 
\end{itemize}

\section*{Acknowledgements}
We want to thank David Cox for his useful comments and Bernard Mourrain for fruitful conversations. 
\bibliography{references}
 \bibliographystyle{abbrv}

\end{document}

%% file: basischoice.tex
%
%
\definecolor{mycolor1}{rgb}{0.00000,0.44700,0.74100}%
\begin{tikzpicture}

\begin{axis}[%
width=1.5in,
height=1.5in,
at={(0.685in,0.241in)},
scale only axis,
xmin=0,
xmax=20,
xlabel = $\deg(x_1)$,
ymin=0,
ymax=20,
ylabel = $\deg(x_2)$,
axis background/.style={fill=white}
]
\addplot [color=blue, draw=none, mark size=2.5pt, mark=*, mark options={solid, blue}, forget plot]
  table[row sep=crcr]{%
0	0\\
1	0\\
0	1\\
2	0\\
1	1\\
0	2\\
3	0\\
2	1\\
1	2\\
0	3\\
4	0\\
3	1\\
2	2\\
1	3\\
0	4\\
5	0\\
4	1\\
3	2\\
2	3\\
1	4\\
0	5\\
6	0\\
5	1\\
4	2\\
3	3\\
2	4\\
1	5\\
0	6\\
7	0\\
6	1\\
5	2\\
4	3\\
3	4\\
2	5\\
1	6\\
0	7\\
8	0\\
7	1\\
6	2\\
5	3\\
4	4\\
3	5\\
2	6\\
1	7\\
0	8\\
9	0\\
8	1\\
7	2\\
6	3\\
5	4\\
4	5\\
3	6\\
2	7\\
1	8\\
0	9\\
9	1\\
8	2\\
7	3\\
6	4\\
5	5\\
4	6\\
3	7\\
2	8\\
1	9\\
9	2\\
8	3\\
7	4\\
6	5\\
5	6\\
4	7\\
3	8\\
2	9\\
9	3\\
8	4\\
7	5\\
6	6\\
5	7\\
4	8\\
3	9\\
9	4\\
8	5\\
7	6\\
6	7\\
5	8\\
4	9\\
9	5\\
8	6\\
7	7\\
6	8\\
5	9\\
9	6\\
8	7\\
7	8\\
6	9\\
9	7\\
8	8\\
7	9\\
9	8\\
8	9\\
9	9\\
};
\addplot [color=black, draw=none, mark size=2.5pt, mark=o, mark options={solid, black}, forget plot]
  table[row sep=crcr]{%
10	0\\
0	10\\
11	0\\
10	1\\
1	10\\
0	11\\
12	0\\
11	1\\
10	2\\
2	10\\
1	11\\
0	12\\
13	0\\
12	1\\
11	2\\
10	3\\
3	10\\
2	11\\
1	12\\
0	13\\
14	0\\
13	1\\
12	2\\
11	3\\
10	4\\
4	10\\
3	11\\
2	12\\
1	13\\
0	14\\
15	0\\
14	1\\
13	2\\
12	3\\
11	4\\
10	5\\
5	10\\
4	11\\
3	12\\
2	13\\
1	14\\
0	15\\
16	0\\
15	1\\
14	2\\
13	3\\
12	4\\
11	5\\
10	6\\
6	10\\
5	11\\
4	12\\
3	13\\
2	14\\
1	15\\
0	16\\
17	0\\
16	1\\
15	2\\
14	3\\
13	4\\
12	5\\
11	6\\
10	7\\
7	10\\
6	11\\
5	12\\
4	13\\
3	14\\
2	15\\
1	16\\
0	17\\
18	0\\
17	1\\
16	2\\
15	3\\
14	4\\
13	5\\
12	6\\
11	7\\
10	8\\
8	10\\
7	11\\
6	12\\
5	13\\
4	14\\
3	15\\
2	16\\
1	17\\
0	18\\
19	0\\
18	1\\
17	2\\
16	3\\
15	4\\
14	5\\
13	6\\
12	7\\
11	8\\
10	9\\
9	10\\
8	11\\
7	12\\
6	13\\
5	14\\
4	15\\
3	16\\
2	17\\
1	18\\
0	19\\
0	0\\
1	0\\
0	1\\
2	0\\
1	1\\
0	2\\
3	0\\
2	1\\
1	2\\
0	3\\
4	0\\
3	1\\
2	2\\
1	3\\
0	4\\
5	0\\
4	1\\
3	2\\
2	3\\
1	4\\
0	5\\
6	0\\
5	1\\
4	2\\
3	3\\
2	4\\
1	5\\
0	6\\
7	0\\
6	1\\
5	2\\
4	3\\
3	4\\
2	5\\
1	6\\
0	7\\
8	0\\
7	1\\
6	2\\
5	3\\
4	4\\
3	5\\
2	6\\
1	7\\
0	8\\
9	0\\
8	1\\
7	2\\
6	3\\
5	4\\
4	5\\
3	6\\
2	7\\
1	8\\
0	9\\
9	1\\
8	2\\
7	3\\
6	4\\
5	5\\
4	6\\
3	7\\
2	8\\
1	9\\
9	2\\
8	3\\
7	4\\
6	5\\
5	6\\
4	7\\
3	8\\
2	9\\
9	3\\
8	4\\
7	5\\
6	6\\
5	7\\
4	8\\
3	9\\
9	4\\
8	5\\
7	6\\
6	7\\
5	8\\
4	9\\
9	5\\
8	6\\
7	7\\
6	8\\
5	9\\
9	6\\
8	7\\
7	8\\
6	9\\
9	7\\
8	8\\
7	9\\
9	8\\
8	9\\
9	9\\
};
\end{axis}

\begin{axis}[%
width=1.5in,
height=1.5in,
at={(3.304in,0.241in)},
scale only axis,
xmin=0,
xmax=20,
xlabel = $\deg(x_1)$,
ymin=0,
ymax=20,
ylabel = $\deg(x_2)$,
axis background/.style={fill=white}
]
\addplot [color=blue, draw=none, mark size=2.5pt, mark=*, mark options={solid, blue}, forget plot]
  table[row sep=crcr]{%
0	12\\
13	0\\
12	1\\
0	7\\
0	8\\
3	0\\
0	3\\
3	1\\
0	4\\
2	1\\
10	0\\
3	10\\
2	11\\
1	12\\
0	13\\
14	0\\
0	11\\
1	0\\
1	6\\
1	2\\
7	0\\
6	0\\
5	0\\
4	0\\
0	5\\
2	0\\
0	2\\
2	12\\
1	13\\
0	14\\
15	0\\
14	1\\
13	2\\
12	3\\
0	0\\
3	2\\
0	6\\
1	10\\
0	9\\
9	0\\
5	10\\
4	11\\
3	12\\
2	13\\
1	14\\
0	15\\
16	0\\
15	1\\
14	2\\
13	3\\
12	4\\
1	1\\
10	6\\
11	1\\
8	8\\
0	10\\
1	11\\
5	11\\
4	12\\
3	13\\
2	14\\
1	15\\
0	16\\
17	0\\
16	1\\
15	2\\
14	3\\
13	4\\
12	5\\
11	6\\
10	7\\
9	8\\
8	9\\
0	1\\
6	11\\
5	12\\
4	13\\
3	14\\
2	15\\
1	16\\
0	17\\
18	0\\
17	1\\
16	2\\
15	3\\
14	4\\
13	5\\
12	6\\
11	7\\
10	8\\
9	9\\
8	10\\
7	11\\
2	10\\
5	13\\
4	14\\
3	15\\
2	16\\
1	17\\
0	18\\
};
\addplot [color=black, draw=none, mark size=2.5pt, mark=o, mark options={solid, black}, forget plot]
  table[row sep=crcr]{%
19	0\\
18	1\\
17	2\\
16	3\\
15	4\\
14	5\\
13	6\\
12	7\\
11	8\\
10	9\\
9	10\\
8	11\\
7	12\\
6	13\\
5	14\\
4	15\\
3	16\\
2	17\\
1	18\\
0	19\\
6	4\\
7	4\\
6	5\\
7	5\\
8	4\\
7	3\\
6	3\\
5	3\\
6	6\\
8	5\\
5	5\\
7	6\\
5	4\\
8	6\\
6	7\\
7	7\\
6	8\\
5	6\\
8	3\\
6	2\\
5	2\\
5	7\\
8	2\\
7	2\\
5	8\\
9	4\\
5	9\\
4	4\\
9	5\\
4	3\\
9	3\\
10	4\\
3	5\\
4	5\\
3	7\\
3	6\\
4	6\\
7	8\\
3	3\\
4	8\\
3	4\\
9	6\\
9	2\\
8	7\\
10	3\\
5	1\\
4	7\\
3	8\\
3	9\\
2	7\\
7	1\\
11	2\\
10	5\\
2	8\\
9	1\\
2	6\\
2	5\\
6	9\\
4	2\\
11	4\\
2	9\\
11	3\\
10	2\\
4	9\\
4	10\\
7	9\\
8	1\\
13	1\\
1	8\\
2	2\\
12	2\\
7	10\\
2	3\\
2	4\\
3	11\\
4	1\\
1	7\\
10	1\\
12	0\\
9	7\\
8	0\\
1	5\\
11	5\\
1	9\\
6	1\\
1	3\\
1	4\\
11	0\\
6	12\\
6	10\\
0	12\\
13	0\\
12	1\\
0	7\\
0	8\\
3	0\\
0	3\\
3	1\\
0	4\\
2	1\\
10	0\\
3	10\\
2	11\\
1	12\\
0	13\\
14	0\\
0	11\\
1	0\\
1	6\\
1	2\\
7	0\\
6	0\\
5	0\\
4	0\\
0	5\\
2	0\\
0	2\\
2	12\\
1	13\\
0	14\\
15	0\\
14	1\\
13	2\\
12	3\\
0	0\\
3	2\\
0	6\\
1	10\\
0	9\\
9	0\\
5	10\\
4	11\\
3	12\\
2	13\\
1	14\\
0	15\\
16	0\\
15	1\\
14	2\\
13	3\\
12	4\\
1	1\\
10	6\\
11	1\\
8	8\\
0	10\\
1	11\\
5	11\\
4	12\\
3	13\\
2	14\\
1	15\\
0	16\\
17	0\\
16	1\\
15	2\\
14	3\\
13	4\\
12	5\\
11	6\\
10	7\\
9	8\\
8	9\\
0	1\\
6	11\\
5	12\\
4	13\\
3	14\\
2	15\\
1	16\\
0	17\\
18	0\\
17	1\\
16	2\\
15	3\\
14	4\\
13	5\\
12	6\\
11	7\\
10	8\\
9	9\\
8	10\\
7	11\\
2	10\\
5	13\\
4	14\\
3	15\\
2	16\\
1	17\\
0	18\\
};
\end{axis}
\end{tikzpicture}%

%% file: condition_numbers_2.tex
%
%
\begin{tikzpicture}

\begin{axis}[%
width=2in,
height=1.6in,
at={(0.758in,0.481in)},
scale only axis,
xmin=1,
xmax=20,
xlabel = degree,
ymode=log,
ymin=1,
ymax=1e+20,
yminorticks=true,
axis background/.style={fill=white},
title style={font=\bfseries},
title={Condition number},
]
\addplot [color=blue, mark size=2.5pt, mark=*, mark options={solid, blue}]
  table[row sep=crcr]{%
1	25.3215672944736\\
2	2466.02281568931\\
3	1153.90755461518\\
4	40219.8729265116\\
5	15438.3324815829\\
6	4608688.61957956\\
7	5545298.43423532\\
8	4871916017.77275\\
9	22953360655962.6\\
10	100760598233.847\\
11	8727752223160.3\\
12	176422160029624\\
13	2.79725264957645e+17\\
14	3.5689169204024e+16\\
15	6.04219278259841e+16\\
16	7.99966638947072e+16\\
17	5.72615852299793e+17\\
18	2.47694181649882e+18\\
19	9.65871090874315e+17\\
20	2.26190462718888e+18\\
};

\addplot [color=red, mark size=2.5pt, mark=*, mark options={solid, red}]
  table[row sep=crcr]{%
1	25.3215672944736\\
2	1219.35521765966\\
3	83.9249208102392\\
4	109.712696191297\\
5	223.747008670015\\
6	279.293930258821\\
7	255.787651543016\\
8	748.128638637579\\
9	751.43928508943\\
10	5858.12339604515\\
11	1119.53166396966\\
12	1283.67721900744\\
13	9074.97629430884\\
14	1157.03592721806\\
15	3071.68207733925\\
16	2740.48769614203\\
17	4807.73071641789\\
18	5922.98224984823\\
19	6678.04750818301\\
20	3829.00442885068\\
};

\end{axis}

\begin{axis}[%
width=2in,
height=1.6in,
at={(3.5in,0.481in)},
scale only axis,
unbounded coords=jump,
xmin=1,
xmax=20,
xlabel = degree,
ymode=log,
ymin=1e-17,
ymax=1,
yminorticks=true,
axis background/.style={fill=white},
title style={font=\bfseries},
title={residual},
legend style={legend cell align=left, align=left, draw=white!15!black}
]
\addplot [color=blue, mark size=2.5pt, mark=*, mark options={solid, blue}]
  table[row sep=crcr]{%
1	4.22686827099288e-17\\
2	6.45498281413041e-14\\
3	2.14993152453622e-13\\
4	2.89973524289785e-12\\
5	1.44750247529949e-10\\
6	1.03133654132114e-07\\
7	1.65740289600093e-07\\
8	0.0152917998674013\\
9	0.093761987748758\\
10	0.0523685333888463\\
11	0.0783667055734093\\
12	0.206546303750913\\
13	0.215684706293172\\
14	0.472487183667272\\
15	inf\\
16	inf\\
17	inf\\
18	inf\\
19	inf\\
20	inf\\
}; \label{block_res_2}

\addplot [color=red, mark size=2.5pt, mark=*, mark options={solid, red}]
  table[row sep=crcr]{%
1	4.22686827099288e-17\\
2	3.8666525535943e-15\\
3	2.6682703509177e-15\\
4	6.53198827998433e-15\\
5	1.8486693091982e-14\\
6	1.2048618895646e-14\\
7	1.35377131015597e-14\\
8	4.80044147129159e-14\\
9	1.16602056639813e-13\\
10	5.27121909306227e-13\\
11	6.65403092730804e-14\\
12	1.19482587405875e-13\\
13	2.81999669478965e-13\\
14	1.54181639312716e-13\\
15	4.80894388262674e-13\\
16	4.61665806373287e-13\\
17	5.82342798149505e-13\\
18	6.43232014631959e-13\\
19	7.63169136883513e-13\\
20	2.9900647312457e-12\\
}; \label{auto_res_2}

\end{axis}
\end{tikzpicture}%

%% file: allresiduals.tex
%
%
\definecolor{mycolor1}{rgb}{0.00000,0.44700,0.74100}%
\definecolor{mycolor2}{rgb}{0.85000,0.32500,0.09800}%
\definecolor{mycolor3}{rgb}{0.92900,0.69400,0.12500}%
\definecolor{mycolor4}{rgb}{0.49400,0.18400,0.55600}%
\definecolor{mycolor5}{rgb}{0.46600,0.67400,0.18800}%
\definecolor{mycolor6}{rgb}{0.30100,0.74500,0.93300}%
\definecolor{mycolor7}{rgb}{0.63500,0.07800,0.18400}%
\begin{tikzpicture}

\begin{axis}[%
width=3.2in,
height=1in,
at={(1.183in,0.482in)},
scale only axis,
xmode=log,
xmin=1,
xmax=65,
xtick = {1,5,10,20,40,60},
log x ticks with fixed point,
xlabel style={font=\color{white!15!black}},
x tick label style={/pgf/number format/1000 sep=\,},
xlabel={degree},
ymode=log,
ymin=1e-17,
ymax=1e-08,
ytick = {1e-15,1e-10},
yminorticks=true,
axis background/.style={fill=white},
title style={font=\bfseries},
title={Residual},
legend style ={font = \tiny},
legend pos=outer north east,
]
\addplot [color=mycolor1, mark size=2pt, mark=*, mark options={solid, mycolor1}]
  table[row sep=crcr]{%
1	1.02372265884633e-17\\
4	5.04673146333949e-15\\
7	5.9903632977243e-14\\
10	2.9217301019183e-14\\
13	3.26446282164688e-13\\
16	4.02277690021145e-14\\
19	1.96787844356176e-13\\
22	7.49707676147145e-13\\
25	3.78491416296196e-13\\
28	2.81228460813394e-13\\
31	6.18487467458757e-12\\
34	7.94282170367728e-12\\
37	1.80988843310861e-11\\
40	1.97374181840396e-12\\
43	2.82839392496916e-12\\
46	2.26946251069382e-12\\
49	1.65203609369149e-12\\
52	2.03971082702321e-11\\
55	2.02748483198475e-11\\
58	8.05219599651947e-12\\
61	2.53686934924506e-11\\
};
\addlegendentry{2}

\addplot [color=mycolor2, mark size=2pt, mark=*, mark options={solid, mycolor2}]
  table[row sep=crcr]{%
1	6.6958096065144e-17\\
3	1.50039646970237e-14\\
5	6.7019285541216e-13\\
7	1.48259213684847e-13\\
9	6.61705678359243e-13\\
11	5.02818224622918e-12\\
13	3.46692240822172e-12\\
15	7.44225742287936e-11\\
17	5.95593050914858e-11\\
19	6.28945830639214e-09\\
21	1.94331528196856e-11\\
};
\addlegendentry{3}

\addplot [color=mycolor3, mark size=2pt, mark=*, mark options={solid, mycolor3}]
  table[row sep=crcr]{%
1	5.20485867378386e-17\\
3	7.29876633220901e-11\\
5	8.10496627742903e-13\\
7	1.99756015684251e-11\\
};
\addlegendentry{4}

\addplot [color=mycolor4, mark size=2pt, mark=*, mark options={solid, mycolor4}]
  table[row sep=crcr]{%
1	7.89028024563525e-17\\
2	2.61465587504515e-15\\
3	1.08919489579191e-13\\
4	4.04186092207232e-11\\
};
\addlegendentry{5}

\addplot [color=mycolor5, mark size=2pt, mark=*, mark options={solid, mycolor5}]
  table[row sep=crcr]{%
1	3.67251406723592e-17\\
2	1.20658463675021e-14\\
3	4.52636507637904e-12\\
};
\addlegendentry{6}

\addplot [color=mycolor6, mark size=2pt, mark=*, mark options={solid, mycolor6}]
  table[row sep=crcr]{%
1	2.15022744815335e-17\\
2	1.33066664045613e-14\\
};
\addlegendentry{7}

\addplot [color=mycolor7, mark size=2pt, mark=*, mark options={solid, mycolor7}]
  table[row sep=crcr]{%
1	5.12175648854036e-17\\
2	5.7112769817212e-13\\
};
\addlegendentry{8}

\end{axis}
\end{tikzpicture}%

%% file: alltimes.tex
%
%
\definecolor{mycolor1}{rgb}{0.00000,0.44700,0.74100}%
\definecolor{mycolor2}{rgb}{0.85000,0.32500,0.09800}%
\definecolor{mycolor3}{rgb}{0.92900,0.69400,0.12500}%
\definecolor{mycolor4}{rgb}{0.49400,0.18400,0.55600}%
\definecolor{mycolor5}{rgb}{0.46600,0.67400,0.18800}%
\definecolor{mycolor6}{rgb}{0.30100,0.74500,0.93300}%
\definecolor{mycolor7}{rgb}{0.63500,0.07800,0.18400}%
\begin{tikzpicture}

\begin{axis}[%
width=3.2in,
height=1in,
at={(0.97in,0.477in)},
scale only axis,
xmode=log,
xmin=1,
xmax=65,
xminorticks=true,
xtick = {1,5,10,20,40,60},
log x ticks with fixed point,
xlabel style={font=\color{white!15!black}},
xlabel={degree},
ymode=log,
ymin=1e-03,
ymax=100000,
yminorticks=true,
axis background/.style={fill=white},
title style={font=\bfseries},
title={Computation Times (seconds)},
legend pos=outer north east,
legend style = {font = \tiny}
]
\addplot [color=mycolor1, mark size=2pt, mark=*, mark options={solid, mycolor1}]
  table[row sep=crcr]{%
1	0.005619\\
4	0.017351\\
7	0.076557\\
10	0.176747\\
13	0.460032\\
16	1.042081\\
19	1.8846\\
22	3.371563\\
25	6.479277\\
28	10.054097\\
31	15.671086\\
34	24.663323\\
37	39.203874\\
40	64.115248\\
43	100.628464\\
46	156.000319\\
49	241.465345\\
52	332.387522\\
55	484.043766\\
58	560.395579\\
61	844.705696\\
};
\addlegendentry{2}

\addplot [color=mycolor2, mark size=2pt, mark=*, mark options={solid, mycolor2}]
  table[row sep=crcr]{%
1	0.007861\\
3	0.065549\\
5	0.676147\\
7	4.581881\\
9	28.575878\\
11	147.562087\\
13	646.011522\\
15	2254.740792\\
17	6822.11999\\
19	17884.330085\\
21	44185.385279\\
};
\addlegendentry{3}

\addplot [color=mycolor3, mark size=2pt, mark=*, mark options={solid, mycolor3}]
  table[row sep=crcr]{%
1	0.003766\\
3	0.975156\\
5	186.081079\\
7	11787.284477\\
};
\addlegendentry{4}

\addplot [color=mycolor4, mark size=2pt, mark=*, mark options={solid, mycolor4}]
  table[row sep=crcr]{%
1	0.002765\\
2	0.51754\\
3	102.112197\\
4	10757.779889\\
};
\addlegendentry{5}

\addplot [color=mycolor5, mark size=2pt, mark=*, mark options={solid, mycolor5}]
  table[row sep=crcr]{%
1	0.002671\\
2	6.649938\\
3	33991.633185\\
};
\addlegendentry{6}

\addplot [color=mycolor6, mark size=2pt, mark=*, mark options={solid, mycolor6}]
  table[row sep=crcr]{%
1	0.003359\\
2	485.640698\\
};
\addlegendentry{7}

\addplot [color=mycolor7, mark size=2pt, mark=*, mark options={solid, mycolor7}]
  table[row sep=crcr]{%
1	0.003589\\
2	34425.610436\\
};
\addlegendentry{8}

\end{axis}
\end{tikzpicture}%

%% file: phcbrt23.tex
%
%
\begin{tikzpicture}[scale = 0.88]

\begin{axis}[%
width=1.924in,
height=1.361in,
at={(0.747in,2.384in)},
scale only axis,
xmin=1,
xmax=25,
xlabel = degree,
ymode=log,
ymin=1e-03,
ymax=10000,
yminorticks=true,
axis background/.style={fill=white},
title style={font=\bfseries},
title={Computation times, $n=2$},
legend style={legend cell align=left, align=left, draw=white!15!black}
]
\addplot [color=blue, mark size=2.5pt, mark=*, mark options={solid, blue}]
  table[row sep=crcr]{%
1	0.06864\\
3	0.075507\\
5	0.116806\\
7	0.203197\\
9	0.312432\\
11	0.551294\\
13	0.720504\\
15	1.175969\\
17	1.919639\\
19	2.735276\\
21	4.517226\\
23	6.09951\\
25	7.795932\\
}; \label{phc}

\addplot [color=red, mark size=2.5pt, mark=*, mark options={solid, red}]
  table[row sep=crcr]{%
1	0.004372\\
3	0.019003\\
5	0.024844\\
7	0.052617\\
9	0.108644\\
11	0.200859\\
13	0.392028\\
15	0.634454\\
17	1.009585\\
19	1.597429\\
21	2.405489\\
23	4.772158\\
25	5.714158\\
}; \label{qr}

\addplot [color=green, mark size=2.5pt, mark=*, mark options={solid, green}]
  table[row sep=crcr]{%
1	0.032115\\
3	0.03066\\
5	0.075706\\
7	0.094997\\
9	1.089896\\
11	0.648798\\
13	3.59569\\
15	5.892224\\
17	25.713985\\
19	87.2742\\
21	90.468217\\
23	320.404604\\
25	483.085849\\
}; \label{brt}

\end{axis}

\begin{axis}[%
width=1.924in,
height=1.361in,
at={(3.279in,2.384in)},
scale only axis,
xmin=1,
xmax=25,
xlabel = degree,
ymode=log,
ymin=1e-18,
ymax=1e-10,
yminorticks=true,
axis background/.style={fill=white},
title style={font=\bfseries},
title={Residuals, $n=2$},
legend style={legend cell align=left, align=left, draw=white!15!black}
]
\addplot [color=blue, mark size=2.5pt, mark=*, mark options={solid, blue}]
  table[row sep=crcr]{%
1	7.71037335237581e-16\\
3	1.81650313467049e-15\\
5	2.66898157846563e-15\\
7	2.16193045160878e-15\\
9	2.65778128496011e-15\\
11	3.87498591067819e-15\\
13	5.52362633878034e-15\\
15	4.77727296024457e-15\\
17	7.23305765667744e-15\\
19	6.93270205768891e-15\\
21	4.27936889766103e-15\\
23	6.35924844492927e-15\\
25	7.06920671149919e-15\\
};

\addplot [color=red, mark size=2.5pt, mark=*, mark options={solid, red}]
  table[row sep=crcr]{%
1	3.7123424077371e-17\\
3	2.06745822745377e-14\\
5	5.9054997446538e-15\\
7	1.62931669094623e-14\\
9	4.05529318596093e-15\\
11	9.08606938205056e-13\\
13	3.45496690855829e-14\\
15	3.44360438302846e-13\\
17	6.01984472292103e-13\\
19	1.82061134067278e-12\\
21	2.07185016428615e-12\\
23	2.13466610023418e-13\\
25	2.58744644262601e-12\\
};

\addplot [color=green, mark size=2.5pt, mark=*, mark options={solid, green}]
  table[row sep=crcr]{%
1	1.51622921944067e-16\\
3	2.05386318758003e-16\\
5	8.3953761404299e-16\\
7	3.52193255369332e-15\\
9	2.97130804771785e-15\\
11	3.85192182960722e-15\\
13	2.28034204523747e-15\\
15	5.23514849906015e-15\\
17	3.69240718343633e-15\\
19	1.82006615673694e-15\\
21	2.72665290402141e-15\\
23	1.8232862928845e-15\\
25	2.69062674596102e-15\\
};

\end{axis}

\begin{axis}[%
width=1.924in,
height=1.361in,
at={(0.747in,0.25in)},
scale only axis,
xmin=1,
xmax=11,
xlabel = degree,
ymode=log,
ymin=1e-04,
ymax=1000,
yminorticks=true,
axis background/.style={fill=white},
title style={font=\bfseries},
title={Computation times, $n=3$},
legend style={legend cell align=left, align=left, draw=white!15!black}
]
\addplot [color=blue, mark size=2.5pt, mark=*, mark options={solid, blue}]
  table[row sep=crcr]{%
1	0.084805\\
3	0.157546\\
5	0.748843\\
7	3.30677\\
9	11.88676\\
11	36.168919\\
};

\addplot [color=red, mark size=2.5pt, mark=*, mark options={solid, red}]
  table[row sep=crcr]{%
1	0.001507\\
3	0.062687\\
5	0.605549\\
7	5.201016\\
9	37.624846\\
11	197.241078\\
};

\addplot [color=green, mark size=2.5pt, mark=*, mark options={solid, green}]
  table[row sep=crcr]{%
1	0.034699\\
3	0.081487\\
5	0.851126\\
7	9.066456\\
9	39.865136\\
11	170.896488\\
};

\end{axis}

\begin{axis}[%
width=1.924in,
height=1.361in,
at={(3.279in,0.25in)},
scale only axis,
xmin=1,
xmax=11,
xlabel = degree,
ymode=log,
ymin=1e-18,
ymax=1e-10,
yminorticks=true,
axis background/.style={fill=white},
title style={font=\bfseries},
title={Residuals, $n=3$},
legend style={legend cell align=left, align=left, draw=white!15!black}
]
\addplot [color=blue, mark size=2.5pt, mark=*, mark options={solid, blue}]
  table[row sep=crcr]{%
1	8.09025321599789e-16\\
3	1.21037351210162e-15\\
5	2.10471242672975e-15\\
7	2.9420311134104e-15\\
9	2.80689621230276e-15\\
11	2.53155278249436e-15\\
};

\addplot [color=red, mark size=2.5pt, mark=*, mark options={solid, red}]
  table[row sep=crcr]{%
1	4.45270915580194e-17\\
3	7.92787905443217e-15\\
5	7.87211062147659e-14\\
7	3.46834594437198e-14\\
9	4.05229529444664e-12\\
11	1.26768598779949e-11\\
};

\addplot [color=green, mark size=2.5pt, mark=*, mark options={solid, green}]
  table[row sep=crcr]{%
1	4.59289129439072e-17\\
3	2.58493596271408e-16\\
5	3.22733191999172e-16\\
7	1.49678520334602e-15\\
9	2.31191915367056e-15\\
11	1.74645003972257e-15\\
};

\end{axis}
\end{tikzpicture}%

%% file: phcbrt45.tex
%
%
\begin{tikzpicture}[scale = 0.88]

\begin{axis}[%
width=1.917in,
height=1.353in,
at={(0.745in,2.372in)},
scale only axis,
xmin=1,
xmax=5,
xlabel = degree,
ymode=log,
ymin=1e-04,
ymax=10000,
yminorticks=true,
axis background/.style={fill=white},
title style={font=\bfseries},
title={Computation times, $n=4$},
legend style={legend cell align=left, align=left, draw=white!15!black}
]
\addplot [color=blue, mark size=2.5pt, mark=*, mark options={solid, blue}]
  table[row sep=crcr]{%
1	0.090602\\
2	0.169293\\
3	0.500047\\
4	2.279334\\
5	9.607632\\
};

\addplot [color=red, mark size=2.5pt, mark=*, mark options={solid, red}]
  table[row sep=crcr]{%
1	0.001569\\
2	0.066506\\
3	0.922148\\
4	20.895306\\
5	228.76498\\
};

\addplot [color=green, mark size=2.5pt, mark=*, mark options={solid, green}]
  table[row sep=crcr]{%
1	0.030487\\
2	0.071725\\
3	0.563942\\
4	3.697541\\
5	25.282906\\
};

\end{axis}

\begin{axis}[%
width=1.917in,
height=1.353in,
at={(3.268in,2.372in)},
scale only axis,
xmin=1,
xmax=5,
xlabel = degree,
ymode=log,
ymin=1e-18,
ymax=1e-10,
yminorticks=true,
axis background/.style={fill=white},
title style={font=\bfseries},
title={Residuals, $n=4$},
legend style={legend cell align=left, align=left, draw=white!15!black}
]
\addplot [color=blue, mark size=2.5pt, mark=*, mark options={solid, blue}]
  table[row sep=crcr]{%
1	4.5938399387003e-16\\
2	1.16808274841526e-15\\
3	1.28594005421631e-15\\
4	1.58407500345178e-15\\
5	1.40468931918487e-15\\
};

\addplot [color=red, mark size=2.5pt, mark=*, mark options={solid, red}]
  table[row sep=crcr]{%
1	3.5689977555947e-17\\
2	8.94949115398091e-15\\
3	2.34281521745936e-14\\
4	4.3099609988503e-13\\
5	6.76404163825655e-13\\
};

\addplot [color=green, mark size=2.5pt, mark=*, mark options={solid, green}]
  table[row sep=crcr]{%
1	8.98718571451837e-17\\
2	1.63617796355172e-16\\
3	5.50443888781577e-16\\
4	3.34420700652991e-16\\
5	1.22352054168695e-15\\
};

\end{axis}

\begin{axis}[%
width=1.917in,
height=1.353in,
at={(0.745in,0.25in)},
scale only axis,
xmin=1,
xmax=3,
xlabel = degree,
xtick = {1,2,3},
ymode=log,
ymin=1e-04,
ymax=10000,
yminorticks=true,
axis background/.style={fill=white},
title style={font=\bfseries},
title={Computation times, $n=5$},
legend style={legend cell align=left, align=left, draw=white!15!black}
]
\addplot [color=blue, mark size=2.5pt, mark=*, mark options={solid, blue}]
  table[row sep=crcr]{%
1	0.071819\\
2	0.266505\\
3	2.608635\\
};

\addplot [color=red, mark size=2.5pt, mark=*, mark options={solid, red}]
  table[row sep=crcr]{%
1	0.001647\\
2	0.457668\\
3	119.099882\\
};

\addplot [color=green, mark size=2.5pt, mark=*, mark options={solid, green}]
  table[row sep=crcr]{%
1	0.036122\\
2	0.251775\\
3	5.100932\\
};

\end{axis}

\begin{axis}[%
width=1.917in,
height=1.353in,
at={(3.268in,0.25in)},
scale only axis,
xmin=1,
xmax=3,
xlabel = degree,
xtick = {1,2,3},
ymode=log,
ymin=1e-18,
ymax=1e-10,
yminorticks=true,
axis background/.style={fill=white},
title style={font=\bfseries},
title={Residuals, $n=5$},
legend style={legend cell align=left, align=left, draw=white!15!black}
]
\addplot [color=blue, mark size=2.5pt, mark=*, mark options={solid, blue}]
  table[row sep=crcr]{%
1	2.41113900069047e-16\\
2	1.00484276118674e-15\\
3	1.26121751036961e-15\\
};

\addplot [color=red, mark size=2.5pt, mark=*, mark options={solid, red}]
  table[row sep=crcr]{%
1	3.34798454251196e-17\\
2	4.9842792688368e-14\\
3	5.17646200852304e-13\\
};

\addplot [color=green, mark size=2.5pt, mark=*, mark options={solid, green}]
  table[row sep=crcr]{%
1	5.60667219746197e-17\\
2	1.83070802513614e-16\\
3	6.1206646032054e-16\\
};

\end{axis}
\end{tikzpicture}%